\documentclass[10pt]{amsart}
\usepackage{graphicx,amscd,color,amsmath,amsfonts,amssymb,geometry,amssymb,paralist}
\newtheorem{theorem}{Theorem}[section]

\newtheorem{corollary}[theorem]{Corollary}

\newtheorem{lemma}[theorem]{Lemma}

\numberwithin{equation}{section}
\input epsf

\DeclareMathOperator{\ext}{ext\,\!}

\DeclareMathOperator{\sign}{sign\,\!}

\begin{document}
\title[Polynomial inequalities on the $\pi/4$--circle sector]{Polynomial inequalities on the $\pi/4$--circle sector}
\author[Ara\'{u}jo \and Jim\'{e}nez \and Mu\~{n}oz \and Seoane]{G. Ara\'{u}jo\textsuperscript{*} \and P. Jim\'{e}nez-Rodr\'iguez\textsuperscript{**} \and G. A. Mu\~{n}oz-Fern\'{a}ndez\textsuperscript{***} \and J. B. Seoane-Sep\'{u}lveda\textsuperscript{***}}
\address{Departamento de An\'{a}lisis Matem\'{a}tico, \newline\indent  Facultad de Ciencias Matem\'{a}ticas, \newline \indent Plaza de Ciencias 3, \newline\indent Universidad Complutense de Madrid,\newline\indent Madrid, 28040 (Spain).}

\email{\newline \indent gdasaraujo@gmail.com \newline \indent pablo.jimenez.rod@gmail.com \newline\indent gustavo\_fernandez@mat.ucm.es \newline\indent jseoane@mat.ucm.es}
\keywords{Bernstein and Markov inequalities, unconditional constants, polarizations constants, polynomial inequalities,
homogeneous polynomials, extreme points.}
\subjclass[2010]{Primary 46G25; Secondary  46B28, 41A44.}
\thanks{\textsuperscript{*}Supported by PDSE/CAPES 8015/14-7.}
\thanks{\textsuperscript{**}The second named author's research was performed during his stay at the Mathematics Department of Kent State University, USA}
\thanks{\textsuperscript{***}Supported by the Spanish Ministry of Science and Innovation, grant MTM2012-34341.}
\begin{abstract}
A number of sharp inequalities are proved for the space ${\mathcal P}\left(^2D\left(\frac{\pi}{4}\right)\right)$ of 2-homogeneous polynomials on ${\mathbb R}^2$ endowed with the supremum norm on the sector $D\left(\frac{\pi}{4}\right):=\left\{e^{i\theta}:\theta\in \left[0,\frac{\pi}{4}\right]\right\}$. Among the main results we can find sharp Bernstein and Markov inequalities and the calculation of the polarization constant and the unconditional constant of the canonical basis of the space ${\mathcal P}\left(^2D\left(\frac{\pi}{4}\right)\right)$.
\end{abstract}
\maketitle

\section{Preliminaries}
The study of low dimensional spaces of polynomials can be an interesting source of examples and counterexamples related to more general questions. In this paper we mind 2-variable, real 2-homogeneous polynomials endowed with the supremum norm on the sector $D\left(\frac{\pi}{4}\right):=\left\{e^{i\theta}:\theta\in \left[0,\frac{\pi}{4}\right]\right\}$. The space of such polynomials is represented by ${\mathcal P}\left(^2D\left(\frac{\pi}{4}\right)\right)$. This paper can be seen as a continuation of \cite{JMPS} and \cite{MPSW}. Other publications in the same spirit can be found in \cite{GMS2,GMSS,MSS_JCA,MSS_MIA,MSS,MRS}.

If $P(x,y)=ax^2+by^2+cxy$, we will often represent $P$ as the point $(a,b,c)$ in ${\mathbb R}^3$. Hence, the norm of ${\mathcal P}\left(^2D\left(\frac{\pi}{4}\right)\right)$ is in fact the norm in ${\mathbb R}^3$ given by
    $$
    \|(a,b,c)\|_{D\left(\frac{\pi}{4}\right)}=\sup\left\{|ax^2+by^2+cxy|:(x,y)\in D\left(\frac{\pi}{4}\right)\right\}.
    $$
In Section \ref{sec:pol_const}, the notation ${\mathcal L}^s\left(^2D\left(\frac{\pi}{4}\right)\right)$ will be useful to represent the symmetric bilinear forms on ${\mathbb R}^2$ endowed with the supremum norm on $D\left(\frac{\pi}{4}\right)$.

In order to obtain sharp polynomial inequalities in ${\mathcal P}\left(^2D\left(\frac{\pi}{4}\right)\right)$ we will use the so called Krein-Milman approach, which is based on the fact that norm attaining convex functions attain their norm at an extreme point of their domain. Hence, an explicit description of the norm $\|\cdot\|_{D\left(\frac{\pi}{4}\right)}$ and the extreme points of the unit ball $B_{D\left({\pi \over 4}\right)}$, denoted by $\ext\left(B_{D\left({\pi \over 4}\right)}\right)$, will be required. Both are presented below:

\begin{lemma}{\cite[Theorem 3.1]{MPSW}} \label{lem_norm}
If $P(x,y)=ax^2+by^2+cxy$, then
    \begin{align*}
    \|P\|_{D\left(\frac{\pi}{4}\right)}&=\begin{cases}
    \max\left\{|a|,\frac{1}{2}|a+b+c|,\frac{1}{2}|a+b+\sign(c)\sqrt{(a-b)^2+c^2}|\right\}&\text{if $c(a-b)\ge 0$,}\\
    \max\{|a|,\frac{1}{2}|a+b+c|\}&\text{if $c(a-b)\leq 0$,}\\
    \end{cases}
    \end{align*}
\end{lemma}

\begin{lemma}{\cite[Theorem 4.4]{MPSW}} \label{lem_ext}
The extreme points of the unit ball of $\mathcal{P}(^2D({\pi \over 4}))$ are given by
    $$
    \ext\left(B_{D\left({\pi \over 4}\right)}\right) = \left\{\pm P_t,\pm Q_s,\pm (1,1,0):\ -1 \leq t \leq 1\text{ and } 1 \leq s \leq 5+4\sqrt{2} \right\},
    $$
where
    \begin{align*}
    P_t:&=(t,4+t+4\sqrt{1+t},-2-2t-4 \sqrt{1+t}),\\
    Q_s:&= (1,s,-2\sqrt{2(1+s)}).
    \end{align*}
\end{lemma}

Let us describe now the three inequalities that will be studied in this paper. Section \ref{sec:Bernstein} is devoted to obtain a Bernstein type inequality for polynomials in $\mathcal{P}\left(^2D\left(\frac{\pi}{4}\right)\right)$. Namely, for a fixed $(x,y)\in D\left(\frac{\pi}{4}\right)$, we find the best (smallest) constant $\Phi(x,y)$ in the inequality
    $$
    \|\nabla P(x,y)\|_2\leq \Phi(x,y)\|P\|_{D\left(\frac{\pi}{4}\right)},
    $$
for all $P\in \mathcal{P}\left(^2D\left(\frac{\pi}{4}\right)\right)$, where $\|\cdot\|_2$ denotes the euclidean norm in ${\mathbb R}^2$. Similarly, we also obtain a Markov global estimate on the gradient of polynomials in $\mathcal{P}\left(^2D\left(\frac{\pi}{4}\right)\right)$, or in other words, the smallest constant $M>0$ in the inequality
    $$
    \|\nabla P(x,y)\|_2\leq M\|P\|_{D\left(\frac{\pi}{4}\right)},
    $$
for all $P\in \mathcal{P}\left(^2D\left(\frac{\pi}{4}\right)\right)$ and $(x,y)\in D\left(\frac{\pi}{4}\right)$. It is necessary to mention that the study of Bernstein and Markov type inequalities has a longstanding tradition. The interested reader can find further information on this classical topic in \cite{BG,Har1,Har,Kre,MR,MN,MSa,Russian,S,Skal,V,W}.

In Section \ref{sec:pol_const} we find the smallest constant $K>0$ in the inequality
    $$
    \|L\|_{D\left(\frac{\pi}{4}\right)}\leq K\|P\|_{D\left(\frac{\pi}{4}\right)},
    $$
where $P$ is an arbitrary polynomial in $\mathcal{P}\left(^2D\left(\frac{\pi}{4}\right)\right)$ and $L\in{\mathcal L}^s\left(^2D\left(\frac{\pi}{4}\right)\right)$ is the polar of $P$. Observe that here $\|L\|_{D\left(\frac{\pi}{4}\right)}$ stands for the sup norm of $L$ over $D\left(\frac{\pi}{4}\right)^2$. Hence, what we do is to provide the polarization constant of the space $\mathcal{P}\left(^2D\left(\frac{\pi}{4}\right)\right)$. The calculation of polarization constants in various polynomial spaces is largely motivated as the  extensive, existing bibliography on the topic shows (see for instance \cite{D,Har1,Martin,IS}).

Finally, in Section \ref{sec:unc_const} we investigate the smallest constant $C>0$ in the inequality
   \begin{equation}\label{equ:unconditionConstant}
   \||P|\|_{D\left(\frac{\pi}{4}\right)}\leq C \|P\|_{D\left(\frac{\pi}{4}\right)},
   \end{equation}
for all $P\in \mathcal{P}\left(^2D\left(\frac{\pi}{4}\right)\right)$, where $|P|$ is the modulus of $P$, i.e., if $P(x,y)=ax^2+by^2+cxy$, then $|P|(x,y)=|a|x^2+|b|y^2+|c|xy$. The constant $C$ turns out to be the unconditional constant of the canonical basis of $\mathcal{P}\left(^2D\left(\frac{\pi}{4}\right)\right)$. It is interesting to note that already in 1914, H. Bohr \cite{B} studied this type of inequalities for infinite complex power series. Actually, the study of Bohr radii is nowadays a fruitful field  (see for instance \cite{BPS,Boas,DF,DGM1,DGM2,DP}). Observe that the relationship between unconditional constants in polynomial spaces and inequalities of the type \eqref{equ:unconditionConstant}  was already noticed in \cite{DGM1}.

\section{Bernstein and Markov-type inequalities for polynomials on sectors}\label{sec:Bernstein}

In this section we provide sharp estimates on the Euclidean length of the gradient $\nabla P$ of a polynomial $P$ in  ${\mathcal P}\left(^2D\left(\frac{\pi}{4}\right)\right)$.

\begin{theorem}
For every $(x,y) \in D\left({\pi \over 4}\right)$ and $P\in{\mathcal P}\left(^2D\left(\frac{\pi}{4}\right)\right)$ we have
    \begin{align*}
    \|\nabla P\|_2\leq \Phi(x,y)\|P\|_{D\left(\frac{\pi}{4}\right)},
    \end{align*}
where
    \begin{align*}
    &\Phi(x,y)\\
    &=\begin{cases}
    4\left[\left(13+8\sqrt{2}\right)x^2+\left(69+48\sqrt{2}\right)y^2-2\left(28+20\sqrt{2}\right)xy\right] & \text{if }0\leq y \leq \frac{\sqrt{2}-1}{2}x \text{ or } \left(4\sqrt{2}-5\right)x\leq y \leq x,\\
    {x^4 \over y^2}+4(x^2+y^2)&\text{if } \frac{\sqrt{2}-1}{2}x \leq y \leq \left(\sqrt{2}-1\right)x,\\
    {\left(3x^2-2xy+3y^2\right)^2 \over 2(x-y)^2}& \text{if } \left(\sqrt{2}-1\right)x\leq y \leq \left(4\sqrt{2}-5\right)x.
    \end{cases}
    \end{align*}

\end{theorem}

\begin{proof}
In order to calculate $\Phi(x,y):=\sup\{\|\nabla P(x,y)\|_{2} : \|P\|_{D\left(\frac{\pi}{4}\right)}\leq 1\}$, by the Krein-Milman approach, it is sufficient to calculate $$\sup\{\|\nabla P(x,y)\|_{2} : P\in\mathrm{ext}(B_{D\left(\frac{\pi}{4}\right)})\}.$$ By symmetry, we may just study the polynomials of Lemma \ref{lem_ext} with positive sign. Let us start first with $P_t(x,y)=tx^2+\left(4+t+4\sqrt{1+t}\right)y^2-2\left(1+t+2\sqrt{1+t}\right)xy$, $t\in[-1,1]$. Then,
$$
\nabla P_t(x,y)=\left(2tx-2\left(1+t+2\sqrt{1+t}\right)y, 2\left(4+t+4\sqrt{1+t}\right)y-2\left(1+t+2\sqrt{1+t}\right)x\right),
$$
so that
\begin{align*}
\|\nabla P_t(x,y)\|_2^2=&4t^2x^2+4\left(1+t+2\sqrt{1+t}\right)^2y^2-8t\left(1+t+2\sqrt{1+t}\right)xy \\
&+4\left(4+t+4\sqrt{1+t}\right)^2y^2+4\left(1+t+2\sqrt{1+t}\right)^2x^2 \\
&-8\left(4+t+4\sqrt{1+t}\right)\left(1+t+2\sqrt{1+t}\right)xy
\end{align*}
Make now the change $u=\sqrt{1+t} \in \left[0,\sqrt{2}\right]$, so that
\begin{align*}
\|\nabla P_u(x,y)\|_2^2=&8(x-y)^2u^4+16\left(x^2-4xy+3y^2\right)u^3\\
&+8\left(x^2-10xy+13y^2\right)u^2+32\left(3y^2-xy\right)u+4\left(x^2+9y^2\right).
\end{align*}
Since
$$
\frac{\partial}{\partial u}\|\nabla P_u(x,y)\|_2^2=16\left(2\left(x-y\right)^2u^2+\left(x^2-8xy+7y^2\right)u+2y\left(3y-x\right)\right)\left(u+1\right),
$$
it follows that the critical points of $\|DP_u(x,y)\|_2^2$ are $u={2y \over x-y},u={3y-x \over 2(x-y)}$ and $u=-1$ if $x\neq y$ and $u=4$ and $u=-1$ if $x=y$. Since we need to consider $0 \leq u \leq \sqrt{2}$, we can directly omit the case $x=y$.

Therefore, we can write
$$
{\partial \over \partial u}\|\nabla P_u(x,y)\|_2^2=32(x-y)^2\left(u-{2y \over x-y}\right)\left(u-{3y-x \over 2(x-y)}\right)(u+1).
$$
Let $u_1={2y \over x-y}$ and $u_2={3y-x \over 2(x-y)}$ (Again, since we need to consider $0 \leq u \leq \sqrt{2}$, we can omit the solution $u=-1$). Also, we have the extra conditions $u_1 \in [0,\sqrt{2}]$ whenever $0\leq y\leq\left(\sqrt{2}-1\right)x$ and $u_2 \in [0, \sqrt{2}]$ whenever $\frac{1}{3}x\leq y\leq\left(4\sqrt{2}-5\right)x$. Considering all these facts, we need to compare the quantities
\begin{equation*} \label{quant1} \begin{aligned}
C_1(x,y):=\|\nabla P_{u_1}(x,y)\|_2^2&=\|\nabla P_{t_1}\|_2^2=4{x^6-4x^5y+7x^4y^2-8x^3y^3+7x^2y^4-4xy^5+y^6 \over (x-y)^4} \\
&=4\left(x^2+y^2\right), \end{aligned}
\end{equation*}
for $0\leq y\leq\left(\sqrt{2}-1\right)x$ and $t_1={3y^2+2xy-x^2 \over (x-y)^2}$,
\begin{equation*} \label{quant2}
\begin{aligned}
C_2(x,y):=\|\nabla P_{u_2}(x,y)\|_2^2&=\|\nabla P_{t_2}\|_2^2={9x^6-30x^5y+55x^4y^2-68x^3y^3+55x^2y^4-30xy^5+9y^6 \over 2(x-y)^4} \\
&={\left(3x^2-2xy+3y^2\right)^2 \over 2(x-y)^2},
\end{aligned}
\end{equation*}
for $\frac{1}{3}x\leq y\leq\left(4\sqrt{2}-5\right)x$ and $t_2={5y^2+2xy-3x^2 \over 4(x-y)^2}$,
\begin{equation*} \label{quant4}
C_3(x,y):=\|\nabla P_{t_3=-1}\|_2^2=4\left(x^2+9y^2\right),
\end{equation*}
and
\begin{equation*} \label{quant3}
C_4(x,y):=\|\nabla P_{t_4=1}\|_2^2=4\left[\left(13+8\sqrt{2}\right)x^2+\left(69+48\sqrt{2}\right)y^2-2\left(28+20\sqrt{2}\right)xy\right].
\end{equation*}

Let us focus now on $Q_s=\left(1,s,-2\sqrt{2(1+s)}\right)$, $1\leq s\leq 5+4\sqrt{2}$. Then, we have
$$
\|\nabla Q_s(x,y)\|_2^2=4x^2+4s^2y^2+8(1+s)(x^2+y^2)-8(1+s)\sqrt{2(1+s)}xy.
$$
Making the change $v=\sqrt{2(1+s)} \in \left[2,2+2\sqrt{2}\right],$ we need to study the function
$$
\|\nabla Q_v(x,y)\|_2^2=v^2\left(y^2v^2-4xyv+4x^2\right)+4\left(x^2+y^2\right).
$$
If $x=y=0$ we have $\|\nabla Q_v(x,y)\|_2^2=0$, so we will assume both $x\neq 0$ and $y\neq 0$. The critical points of $\|\nabla Q_v(x,y)\|_2^2$ are $v={x \over y},v={2x \over y}$ and $v=0$ (but $0\notin [2,2+2\sqrt{2}]$). Observe that $v_1=\frac{x}{y}\in \left[2,2+2\sqrt{2}\right]$ whenever $\frac{\sqrt{2}-1}{2}x\leq y\leq \frac{1}{2}x$ and $v_2=\frac{2x}{y}\in \left[2,2+2\sqrt{2}\right]$ whenever $y \geq \left(\sqrt{2}-1\right)x$. Thus, we also need to compare the quantities
$$
C_5(x,y):=\|\nabla Q_{v_1}(x,y)\|_2^2=\|\nabla Q_{s_1}(x,y)\|_2^2={x^4 \over y^2}+4\left(x^2+y^2\right),
$$
for $\frac{\sqrt{2}-1}{2}x\leq y\leq \frac{1}{2}x$ and $s_1=\frac{x^2-2y^2}{2y^2}$,
$$
C_6(x,y):=\|\nabla Q_{v_2}(x,y)\|_2^2=\|\nabla Q_{s_2}(x,y)\|_2^2=4\left(x^2+y^2\right),
$$
for $\left(\sqrt{2}-1\right)x\leq y \leq x$ and $s_2=\frac{2x^2-y^2}{y^2}$, and also
$$
C_7(x,y):=\|\nabla Q_{s_3=1}\|_2^2 =4\left(x^2+y^2\right)+16(x-y)^2,
$$
and
\begin{align*}
C_8(x,y)&:=\|\nabla Q_{s_4=5+4\sqrt{2}}\|_2^2 \\
        & =\left(12+8\sqrt{2}\right)\left[4x^2+\left(12+8\sqrt{2}\right)y^2-\left(8+8\sqrt{2} \right)xy\right]+4\left(x^2+y^2\right) \\
        & = 4\left[\left(13+8\sqrt{2}\right)x^2+\left(69+48\sqrt{2}\right)y^2-2\left(28+20\sqrt{2}\right)xy\right].
\end{align*}
Note that (the reader can take a look at Figures \ref{graph1}, \ref{graph2} and \ref{graph3})
\begin{align*}
& C_1(x,y),C_6(x,y)\leq C_7(x,y)\leq \left\{
	\begin{array}{ll}
		C_4(x,y) & \text{if }0\leq y\leq \frac{2-\sqrt{2}}{2}x \text{ or }\frac{1}{2}x\leq y\leq x, \\
		C_5(x,y) & \text{if }\frac{\sqrt{2}-1}{2}x\leq y\leq \frac{1}{2}x,
	\end{array}\right. \\
& C_3(x,y)\leq \left\{
	\begin{array}{ll}
		C_2(x,y) & \text{if }\frac{1}{3}x\leq y\leq\left(4\sqrt{2}-5\right)x, \\
		C_4(x,y) & \text{if }0\leq y\leq\frac{1}{3}x\text{ or }(4\sqrt{2}-5)x\leq y\leq x,
	\end{array}\right. \\
& C_8(x,y) = C_4(x,y).
\end{align*}
Hence, for $(x,y) \in D\left({\pi \over 4}\right)$,
\begin{align*}
\Phi(x,y) & =\sup\left\{\| \nabla P(x,y) \|_2 \, : \, P \in  \ext\left(B_{D\left({\pi \over 4}\right)}\right) \right\} \\
          &=\begin{cases}
          C_4(x,y) & \text{if $0\leq y \leq \frac{\sqrt{2}-1}{2}x$ or $(4\sqrt{2}-5)x\leq y\leq x$,}\\
		  C_5(x,y)&\text{if $\frac{\sqrt{2}-1}{2}x\leq y\leq (\sqrt{2}-1)x$,}\\
		  C_2(x,y)& \text{if $(\sqrt{2}-1)x\leq y \leq (4\sqrt{2}-5)x$}.
\end{cases}
\end{align*}
In order to illustrate the previous step, the reader can take a look at Figure \ref{graph4}.
\end{proof}

\begin{figure}
\centering
\includegraphics[width=0.9\textwidth]{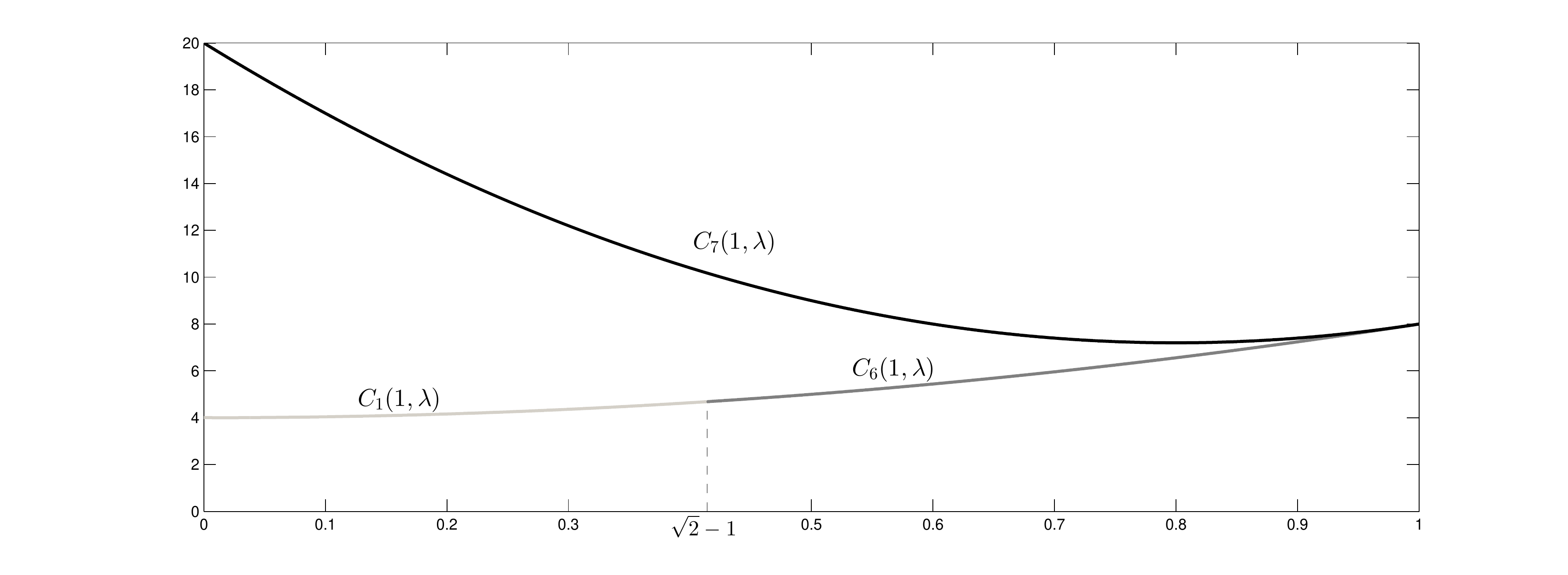}
\caption{Graphs of the mappings $C_1(1,\lambda)$, $C_6(1,\lambda)$, $C_7(1,\lambda)$.}\label{graph1}
\end{figure}

\begin{figure}
\centering
\includegraphics[width=0.9\textwidth]{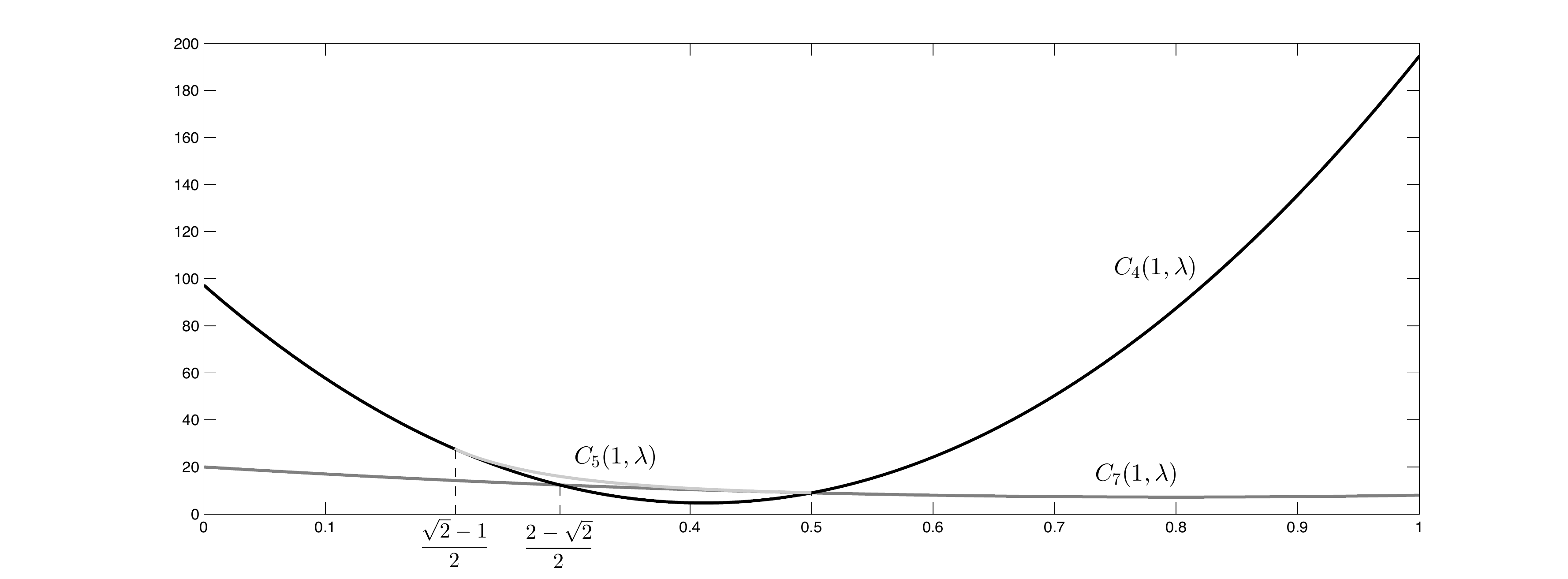}
\caption{Graphs of the mappings $C_4(1,\lambda)$, $C_5(1,\lambda)$, $C_7(1,\lambda)$.}\label{graph2}
\end{figure}

\begin{figure}
\centering
\includegraphics[width=0.9\textwidth]{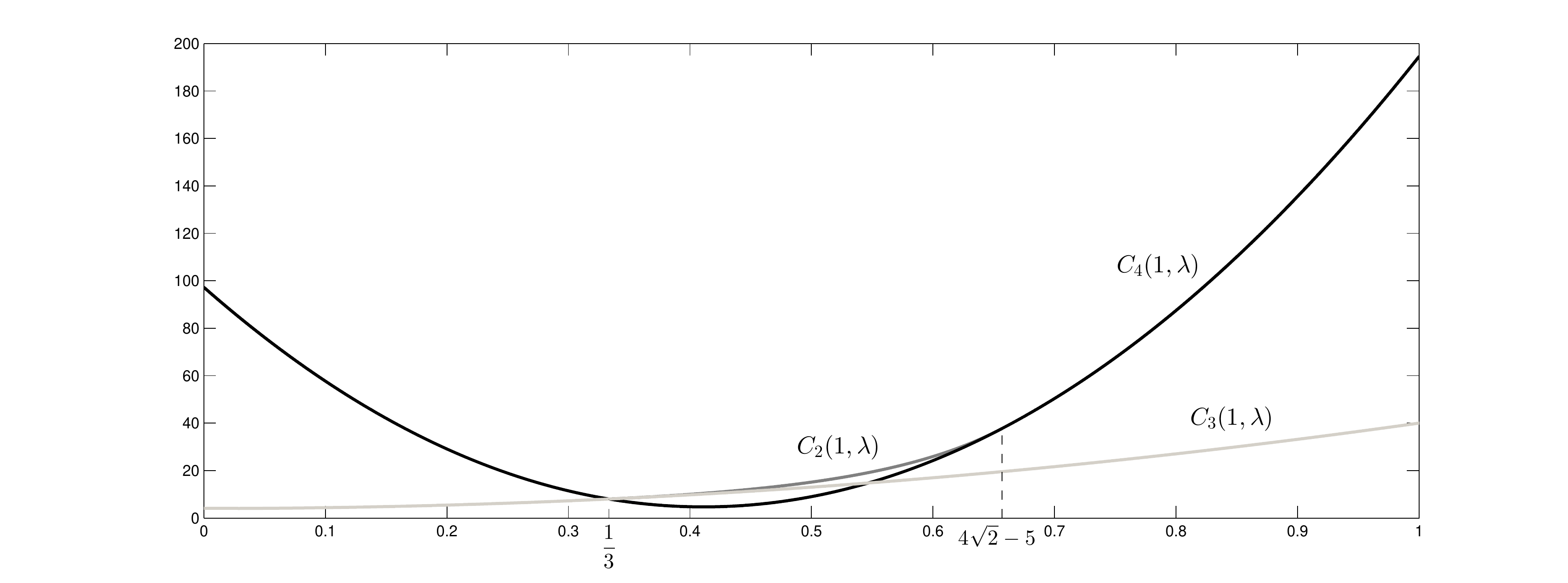}
\caption{Graphs of the mappings $C_2(1,\lambda)$, $C_3(1,\lambda)$, $C_4(1,\lambda)$.}\label{graph3}
\end{figure}

\begin{figure}
\centering
\includegraphics[width=0.9\textwidth]{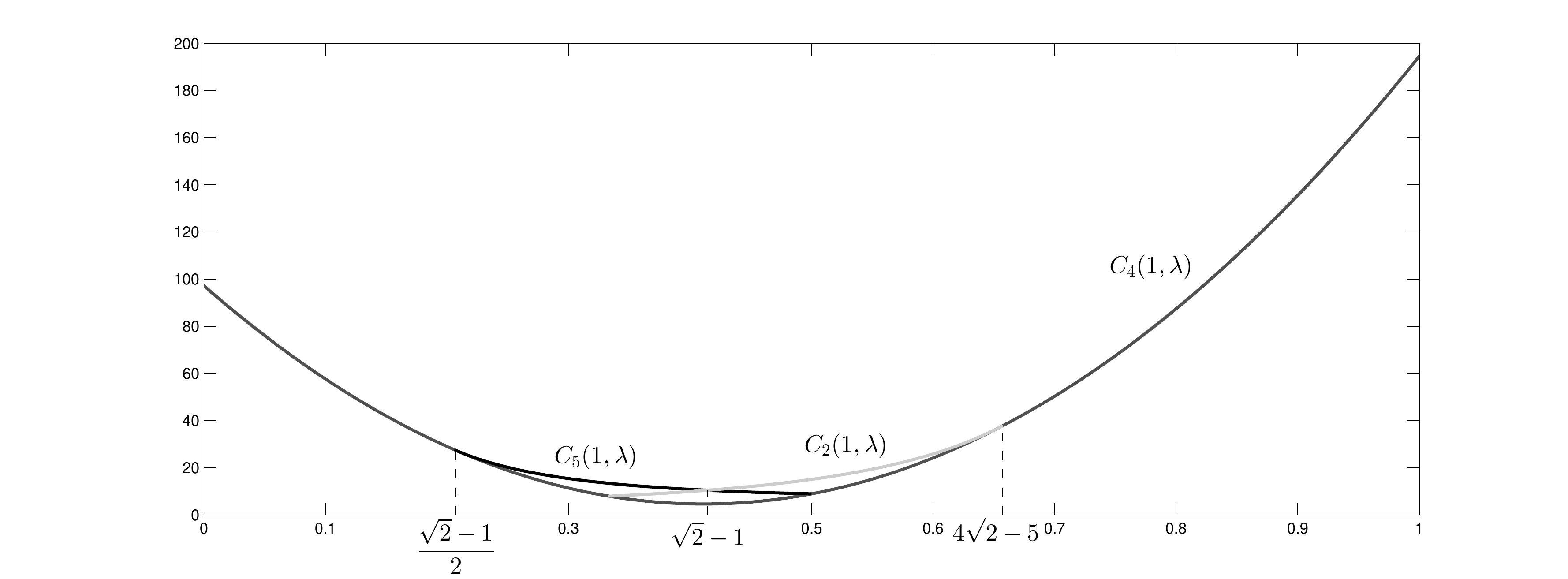}
\caption{Graphs of the mappings $C_2(1,\lambda)$, $C_4(1,\lambda)$, $C_5(1,\lambda)$.}\label{graph4}
\end{figure}

\begin{corollary}
If  $P\in{\mathcal P}\left(D\left(\frac{\pi}{4}\right)\right)$, then
    $$
    \sup\left\{\|\nabla P(x,y)\|_2:(x,y)\in D\left(\frac{\pi}{4}\right)\right\}\leq 4(13+8\sqrt{2})\|P\|_{D\left(\frac{\pi}{4}\right)},
    $$
with equality for the polynomials $P_1(x,y)=\pm\left(x^2+(5+4\sqrt{2})y^2-2(2+2\sqrt{2})xy\right)$.
\end{corollary}


\section{Polarization constants for polynomials on sectors}\label{sec:pol_const}

In this section we find the exact value of the polarization constant of the space ${\mathcal P}\left(^2D\left(\frac{\pi}{4}\right)\right)$. In order to do that, we prove a Bernstein type inequality for polynomials in ${\mathcal P}\left(^2D\left(\frac{\pi}{4}\right)\right)$. Observe that if  $P\in {\mathcal P}\left(^2D\left(\frac{\pi}{4}\right)\right)$ and $(x,y)\in D\left(\frac{\pi}{4}\right)$ then the differential $DP(x,y)$ of $P$ at $(x,y)$ can be viewed as a linear form. What we shall do is to find the best estimate for $\|DP(x,y)\|_{D\left(\frac{\pi}{4}\right)}$ (the sup norm of $DP(x,y)$ over the sector $D\left(\frac{\pi}{4}\right)$) in terms of $(x,y)$ and $\|P\|_{D\left(\frac{\pi}{4}\right)}$. First,
we state a lemma that will be useful in the future:

\begin{lemma} \label{lem_aux}
Let $a,b \in \mathbb{R}$. Then,
\begin{align*}
\sup_{\theta \in \left[0, {\pi \over 4}\right]}|a\cos \theta+b\sin \theta|&=\begin{cases}\max\left\{|a|,{\sqrt{2}\over 2}|a+b|\right\} & \text{if ${b \over a}>1$ or ${b \over a} <0$}, \\
\sqrt{a^2+b^2} & \text{otherwise.}
\end{cases} \\
&=\begin{cases} \sqrt{a^2+b^2} & \text{if $0<{b \over a}<1$}, \\
{\sqrt{2} \over 2}|a+b| & \text{if $\left(1-\sqrt{2}\right)b<a<b$ or $b<a<\left(1-\sqrt{2}\right)b$},\\
|a| & \text{if $-\left(1+\sqrt{2}\right)a<b<0$ or $0<b<-\left(1+\sqrt{2}\right)a$}.
\end{cases}
\end{align*}

\end{lemma}

\begin{theorem}
For every $(x,y)\in D(\frac{\pi}{4})$ and $P\in {\mathcal P}(^2D(\frac{\pi}{4}))$ we have that
    \begin{equation}\label{equ:Phipi4}
    \|DP(x,y)\|_{D(\frac{\pi}{4})}\leq \Psi(x,y)\|P\|_{D(\frac{\pi}{4})},
    \end{equation}
where
$$
\Psi(x,y)=\left\{ \begin{array}{ll} \sqrt{2}\left[\left(1+2\sqrt{2}\right)x-\left(3+2\sqrt{2}\right)y\right] & \text{if }0 \leq y < {2\sqrt{2}-1 \over 7}x, \\
{\sqrt{2}(x^2+3y^2) \over 2y} & \text{if }{2\sqrt{2}-1 \over 7}x \leq y < (\sqrt{2}-1) x, \\
2\left(x+{y^2 \over x-y}\right) & \text{if }(\sqrt{2}-1) x \leq y < \left(2-\sqrt{2}\right)x, \\
4\left(1+\sqrt{2}\right)y-2x & \text{if }\left(2-\sqrt{2}\right)x \leq y \leq x
\end{array} \right.
$$
Moreover, inequality \eqref{equ:Phipi4} is optimal for each $(x,y)\in D(\frac{\pi}{4})$.
\end{theorem}

\begin{proof}
In order to calculate $\Psi(x,y):=\sup\{\|DP(x,y)\|_{D(\frac{\pi}{4})}:\ \|P\|_{D({\pi \over 4})})\leq 1\}$, by the Krein-Milman approach, it suffices to calculate
    $$
    \sup \{\|D P(x,y) \|_{D({\pi \over 4})} \, : \, P \in \text{ext}(B_{D({\pi \over 4})})\}.
    $$
By symmetry, we may just study the polynomials of Lemma \ref{lem_ext} with positive sign. Let us start first with
$$
P_t(x,y)=tx^2+\left(4+t+4 \sqrt{1+t}\right)y^2-\left(2+2t+4 \sqrt{1+t} \right)xy.
$$
So we may write
$$
\nabla P_t(x,y)=\left(2tx-\left(2+2t+4 \sqrt{1+t}\right)y, \, 2 \left(4+t+4 \sqrt{1+t}\right)y-\left(2+2t+4\sqrt{1+t}\right)x \right),
$$
from which
$$
\begin{aligned}
\| D P_t(x,y) \|_{D({\pi \over 4})} &= \sup_{0 \leq \theta \leq {\pi \over 4}} \left|2 \left[ tx-\left(1+t+2\sqrt{1+t} \right)y \right] \cos \theta \right.\\
& \left. \quad + \, 2 \left[\left(4+t+4 \sqrt{1+t}\right)y-\left(1+t+2\sqrt{1+t} \right)x \right] \sin \theta \right| \\
&=2x \sup_{0 \leq \theta \leq {\pi \over 4}}|f_{\lambda}(t, \theta)|, \\
& \begin{aligned} \quad \text{for }f_{\lambda}(t, \theta)&=\left[t-\left(1+t+2\sqrt{1+t} \right) \lambda \right] \cos \theta \\
& \quad +\left[\left(4+t+4\sqrt{1+t} \right) \lambda-\left(1+t+2 \sqrt{1+t} \right) \right] \sin \theta, \end{aligned}
\end{aligned}
$$
where $\lambda={y \over x}, \, x \neq 0$ (the case $x=0$ is trivial, since the only point in $D({\pi \over 4})$ where $x=0$ is $(0,0)$, in which case $P_t(0,0)=\|DP_{t}(0,0)\|_{D\left(\frac{\pi}{4}\right)}=0$). \\
\noindent
We  need to calculate
$$
\sup_{-1 \leq t \leq 1} \| D P_t(x,y) \|_{D({\pi \over 4})}=2x \sup_{\substack{0 \leq \theta \leq {\pi \over 4} \\ -1 \leq t \leq 1}} |f_{\lambda}(t, \theta)|.
$$
Let us define $C_1= [-1,1] \times [0,{\pi \over 4}] $. We will analyze $5$ cases.

\medskip

(1) $(t, \theta) \in (-1,1) \times (0,{\pi \over 4})$.

\medskip

We are interested just in critical points. Hence,

\begin{equation} \label{eq1}
\begin{aligned}
{\partial f_{\lambda} \over \partial t}(t, \theta)&=\left[ \left(1+{2 \over \sqrt{1+t}} \right) \lambda-\left(1+{1 \over \sqrt{1+t}} \right) \right] \sin \theta \\
 & \quad + \left[1-\left(1+{1 \over \sqrt{1+t}} \right) \lambda \right] \cos \theta=0,
\end{aligned}
\end{equation}

\begin{equation} \label{eq2}
\begin{aligned}
{\partial f_{\lambda} \over \partial \theta}(t, \theta)&=\left[\left(1+t+2\sqrt{1+t} \right) \lambda-t\right] \sin \theta \\
& \quad + \left[ \left(4+t+4 \sqrt{1+t} \right) \lambda-\left(1+t+2 \sqrt{1+t} \right) \right] \cos \theta=0
\end{aligned}
\end{equation}

Equation \eqref{eq2} tells us that

\begin{equation} \label{eq3}
\sin \theta={\left(4+t+4 \sqrt{1+t} \right) \lambda-\left(1+t+2\sqrt{1+t} \right) \over t-\left(1+t+2\sqrt{1+t} \right) \lambda} \cos \theta.
\end{equation}

If we now plug \eqref{eq3} in equation \eqref{eq1}, we obtain

$$
\begin{aligned}
0 &= \left\{ \left[ 1-\left( 1 +{1 \over \sqrt{1+t}} \right) \lambda \right]+\left[\left(1+{2 \over \sqrt{1+t}} \right) \lambda-\left(1+{1 \over \sqrt{1+t}} \right) \right]\right. \\
& \left. \quad \times {\left(4+t+4\sqrt{1+t} \right) \lambda-\left(1+t+2 \sqrt{1+t} \right) \over t-\left(1+t+2 \sqrt{1+t} \right) \lambda} \right\} \cos \theta.
\end{aligned}
$$

Using that $0< \theta < {\pi \over 4}$, we can conclude

$$
\begin{aligned}
0 &=  \left[ 1-\left( 1 +{1 \over \sqrt{1+t}} \right) \lambda \right]+\left[\left(1+{2 \over \sqrt{1+t}} \right) \lambda-\left(1+{1 \over \sqrt{1+t}} \right) \right]  \\
&  \quad \times {\left(4+t+4\sqrt{1+t} \right) \lambda-\left(1+t+2 \sqrt{1+t} \right) \over t-\left(1+t+2 \sqrt{1+t} \right) \lambda}
\end{aligned}
$$
and thus
$$
\begin{aligned}
0&=\left[1-\left(1+{1 \over \sqrt{1+t}} \right) \lambda \right]\cdot \left[t-\left(1+t+2 \sqrt{1+t} \right) \lambda \right] \\
& \quad + \left[ \left(1+{2 \over \sqrt{1+t}} \right) \lambda-\left(1+{1 \over \sqrt{1+t}} \right) \right] \cdot \left[ \left( 4+t+4 \sqrt{1+t} \right) \lambda-\left(1+t+2 \sqrt{1+t} \right) \right] \\
&=t-\left(1+t+2 \sqrt{1+t} \right) \lambda-t\lambda+\left(1+t+2 \sqrt{1+t} \right) \lambda^2-{\lambda t \over \sqrt{1+t}} \\
& \quad  + {\lambda^2 \over \sqrt{1+t}} \left(1+t+2\sqrt{1+t} \right) + \left(1+{2 \over \sqrt{1+t}} \right) \left(4+t+4\sqrt{1+t} \right) \lambda^2 \\
& \quad -\left( 1+{2 \over \sqrt{1+t}} \right) \left(1+t+2 \sqrt{1+t} \right) \lambda - \left(1+{1 \over \sqrt{1+t}} \right) \left(4+t+4\sqrt{1+t} \right) \lambda \\
 & \quad + \left(1+{1 \over \sqrt{1+t}} \right) \left(1+t+2 \sqrt{1+t} \right) \\
&=t\left(1-2\lambda+2\lambda^2-2\lambda+1\right)+  \left(-2\lambda+2\lambda^2+4\lambda^2-2\lambda-4\lambda+2\right)\sqrt{1+t} \\
& \quad +{t \over \sqrt{1+t}}\left(-\lambda+\lambda^2+2\lambda^2-2\lambda-\lambda+1\right)+{1 \over \sqrt{1+t}}\left(\lambda^2+8\lambda^2-2\lambda-4\lambda+1\right) \\
 & \quad + \left(-\lambda+\lambda^2+2\lambda^2+4 \lambda^2-\lambda-4\lambda+1+2+8\lambda^2-8\lambda\right) \\
&=2t(\lambda-1)^2+6\sqrt{1+t}(\lambda-1) \left( \lambda-{1 \over 3}\right)+3{t \over \sqrt{1+t}}(\lambda-1)\left(\lambda-{1 \over 3} \right) \\
& \quad + {1 \over \sqrt{1+t}}(3\lambda-1)^2+15\left(\lambda-\frac{1}{3}\right) \left(\lambda-{3 \over 5} \right).
\end{aligned}
$$
Working with this last expression, we get
\begin{eqnarray*}
0 &=& 2t\sqrt{1+t}(\lambda - 1)^2+6(1+t)(\lambda-1)\left(\lambda-{1 \over 3} \right)+3t(\lambda-1) \left(\lambda-{1 \over 3} \right) \\
& &+ (3 \lambda-1)^2 + 15 \sqrt{1+t}\left(\lambda-\frac{1}{3}\right) \left( \lambda-{3 \over 5} \right)
\end{eqnarray*}
and hence, rearranging terms,
\begin{equation} \label{eq4}
\sqrt{1+t} \left[15\left(\lambda-\frac{1}{3}\right) \left( \lambda-{3 \over 5} \right)+2t(\lambda-1)^2 \right]=-9t(\lambda-1) \left( \lambda-{1 \over 3} \right)-15 \left( \lambda-{1 \over 3} \right)\left(\lambda-{3 \over 5}\right).
\end{equation}
If $\lambda=1$, we obtain
$$
\sqrt{1+t}+1=0
$$
and so, in particular, we have $\lambda\neq1$. Equation \eqref{eq4} has two solutions,
$$
t_1(\lambda)={-1+2\lambda+3\lambda^2 \over (\lambda-1)^2} \quad \text{and} \quad t_2(\lambda)={5 \lambda^2+2\lambda-3 \over 4(\lambda-1)^2}.
$$
Using equation \eqref{eq1}, we may see
$$
\tan \theta={\left(1+{1 \over \sqrt{1+t}} \right)\lambda-1 \over \left(1+{2 \over \sqrt{1+t}} \right)\lambda-\left(1+{1 \over \sqrt{1+t}} \right)}.
$$
In particular, evaluating in $t_1(\lambda)$ we obtain
$$
\tan\theta_{1}=\frac{\left(  1+\frac{1-\lambda}{2\lambda}\right)  \lambda
-1}{\left(  1+\frac{1-\lambda}{\lambda}\right)  \lambda-\left(  1+\frac
{1-\lambda}{2\lambda}\right)  }=\lambda,
$$
in which case we have
$$
D_{1,1}(\lambda):=\left|f_{\lambda}(t_1,\theta_1)\right|=\left|-\sqrt{1+\lambda^2}\right|=\sqrt{1+\lambda^2}.
$$

Regarding $t_{2}\left(  \lambda\right)  $, we obtain
$$
\tan\theta_{2}=\frac{\left(  1+\sqrt{\frac{4\left(  \lambda-1\right)  ^{2}%
}{\left(  3\lambda-1\right)  ^{2}}}\right)  \lambda-1}{\left(  1+2\sqrt
{\frac{4\left(  \lambda-1\right)  ^{2}}{\left(  3\lambda-1\right)  ^{2}}%
}\right)  \lambda-\left(  1+\sqrt{\frac{4\left(  \lambda-1\right)  ^{2}%
}{\left(  3\lambda-1\right)  ^{2}}}\right)  }.
$$
Since $\theta_2 \in \left(0,{\pi \over 4}\right)$, we need to guarantee $0<\tan \theta_2 < 1$, and for this we need  $0<\lambda<{1 \over 5}$.
Therefore $$\tan\theta_{2}=\frac{5\lambda-1}{7\lambda-3}$$ and in this case,
$$
\begin{aligned}
D_{1,2}(\lambda)&:=\left|f_{\lambda}(t_2,\theta_2)\right|\\
&=\left|\left[{5\lambda^2+2\lambda-3 \over 4(\lambda-1)^2}-\left({9\lambda^2-6\lambda+1 \over 4(\lambda-1)^2}+{3\lambda-1 \over \lambda-1}\right)\lambda\right]{3-7\lambda \over \sqrt{74\lambda^2-52\lambda+10}}\right. \\
& \quad \left.+\left[\left(3+{9\lambda^2-6\lambda+1 \over 4(\lambda-1)^2}+{6\lambda-2 \over \lambda-1}\right)\lambda-\left({9\lambda^2-6\lambda+1 \over 4(\lambda-1)^2}+{3\lambda-1 \over \lambda-1}\right)\right]{1-5\lambda \over \sqrt{74\lambda^2-52\lambda+10}}\right| \\
&=\left|-{78\lambda^4-208\lambda^3+196\lambda^2-80\lambda+14 \over 4(\lambda-1)^2\sqrt{74\lambda^2-52\lambda+10}}\right| \\
&=\left|-{39\lambda^2-26\lambda+7 \over 2\sqrt{74\lambda^2-52\lambda+10}}\right| \\
&={39\lambda^2-26\lambda+7 \over 2\sqrt{74\lambda^2-52\lambda+10}}.
\end{aligned}
$$

\medskip

(2) $\theta=0, -1 \leq t \leq 1$.

\medskip

We have
$$
f_{\lambda}(t,0)=t-\left(1+t+2\sqrt{1+t} \right)\lambda.
$$
Then,
$$
\begin{aligned}
&f_{\lambda}(-1,0)=-1, \\
&f_{\lambda}(1,0)=1-2\left( 1+\sqrt{2} \right) \lambda,
\end{aligned}
$$
and hence
$$
|f_{\lambda}(1,0)|=\left\{
\begin{array}{ll}
1-2(1+\sqrt{2})\lambda & \text{if }0 \leq \lambda <{\sqrt{2}-1 \over 2}, \\
2\left(1+\sqrt{2} \right) \lambda-1 & \text{if }{\sqrt{2}-1 \over 2} \leq \lambda \leq 1.
\end{array}\right.
$$
Working now on $(-1,1)$, since
$$
f_{\lambda}'(t,0)=1-\left(1+{1 \over \sqrt{1+t}} \right)\lambda,
$$
the critical point of $f_{\lambda}(t,0)$ is
$$
t ={\lambda^2 \over (1-\lambda)^2}-1.
$$
Recall that we need to make sure that $-1<t<1$. Therefore, in this case we also need to ask
$$
\lambda <{\sqrt{2} \over 1+\sqrt{2}}=2-\sqrt{2}.
$$

Plugging the critical point of $f_{\lambda}(t,0)$ into $f_{\lambda}(t,0)$, we obtain
$$
f_{\lambda} \left({\lambda^2 \over (\lambda-1)^2}-1,0 \right)={\lambda^2 \over (\lambda-1)^2}-1-\left[{\lambda^2 \over (\lambda-1)^2}+{2 \lambda \over 1-\lambda} \right] \lambda={\lambda^2 \over \lambda-1}-1,
$$
and hence
$$
\left| f_{\lambda} \left({\lambda^2 \over (\lambda-1)^2}-1,0 \right)\right|=1+{\lambda^2 \over 1-\lambda}.
$$

\begin{itemize}
\item Assume first $0 \leq \lambda <{\sqrt{2}-1 \over 2}$. Then,
$$
\sup_{-1 \leq t \leq 1}|f_{\lambda}(t,0)|=\max\left\{1, \, 1-2\left(1+\sqrt{2} \right) \lambda, \, 1+{\lambda^2 \over 1-\lambda} \right\}=1+{\lambda^2 \over 1-\lambda}.
$$
\item Assume now ${\sqrt{2}-1 \over 2} \leq \lambda < 2-\sqrt{2}$. Then,
$$
\sup_{-1 \leq t \leq 1}|f_{\lambda}(t,0)|=\max\left\{1, \, 2 \left(1+\sqrt{2} \right) \lambda-1, \, 1+{\lambda^2 \over 1-\lambda} \right\}=1+{\lambda^2 \over 1-\lambda}.
$$
\item Assume finally $2-\sqrt{2} \leq \lambda \leq 1$. Then,
$$
\sup_{-1 \leq t \leq 1}|f_{\lambda}(t,0)|=\max\left\{1, \, 2 \left(1+\sqrt{2} \right) \lambda-1 \right\}=2 \left(1+\sqrt{2} \right) \lambda-1.
$$
\end{itemize}
So, in conclusion,
$$
\begin{array}{lcl}
\sup\limits_{-1 \leq t \leq 1}|f_{\lambda}(t,0)| & =  & \left\{ \begin{array}{ll} 1+{\lambda^2 \over 1-\lambda} & \text{if }0 \leq \lambda < 2-\sqrt{2}, \vspace{0.2cm}\\
\left(2+2\sqrt{2} \right) \lambda-1 & \text{if }2-\sqrt{2} \leq \lambda \leq 1, \end{array} \right. \vspace{0.2cm} \\
                                          & =: & \left\{ \begin{array}{ll} D_{2,1}(\lambda) & \text{if }0 \leq \lambda < 2-\sqrt{2}, \vspace{0.2cm} \\
D_{2,2}(\lambda) & \text{if }2-\sqrt{2} \leq \lambda \leq 1. \end{array} \right.
\end{array}
$$

\medskip

(3) $\theta={\pi \over 4}$ and $-1 \leq t \leq 1$.

\medskip

We have
$$
\begin{aligned}
f_{\lambda}\left(t,{\pi \over 4}\right)&={\sqrt{2} \over 2} \left[t-\left(1+t+2\sqrt{1+t} \right) \lambda+\left(4+t+4\sqrt{1+t} \right) \lambda-\left(1+t+2\sqrt{1+t} \right) \right] \\
&={\sqrt{2} \over 2}\left[ \left(3+2\sqrt{1+t} \right) \lambda-\left(1+2\sqrt{1+t} \right) \right].
\end{aligned}
$$
Again, we have
$$
\begin{aligned}
f_{\lambda}\left(-1,{\pi \over 4}\right)&=\frac{\sqrt{2}}{2}\left(3\lambda-1\right), \\
f_{\lambda}\left(1,{\pi \over 4}\right)&=\frac{\sqrt{2}}{2}\left[\left(3+2\sqrt{2} \right) \lambda-\left(1+2 \sqrt{2} \right)\right], \\
f_{\lambda}'\left(t,{\pi \over 4}\right)&=\frac{\sqrt{2}}{2}\left[{\lambda \over \sqrt{1+t}}-{1 \over \sqrt{1+t}}\right].
\end{aligned}
$$
and $f_{\lambda}'(t,{\pi \over 4})=0$ implies $\lambda=1$ (in which case $f_{\lambda}(t,{\pi \over 4})=\sqrt{2}$ for every $t$).

\begin{itemize}
\item Assume first $0 \leq \lambda < {1 \over 3}$. Then,
$$
\begin{aligned}
\sup_{-1 \leq t \leq 1} |f_{\lambda}\left(t,{\pi \over 4}\right)| &={\sqrt{2} \over 2} \max\left\{ \left(1+2\sqrt{2} \right)-\left(3+2\sqrt{2} \right) \lambda, \, 1-3 \lambda \right\} \\
&={\sqrt{2} \over 2} \left[  \left(1+2\sqrt{2} \right)-\left(3+2\sqrt{2} \right) \lambda \right]
\end{aligned}
$$
\item Assume now ${1 \over 3} \leq \lambda < 4\sqrt{2}-5$. Then,
$$
\begin{aligned}
\sup_{-1 \leq t \leq 1} |f_{\lambda}\left(t,{\pi \over 4}\right)| &={\sqrt{2} \over 2} \max\left\{ \left(1+2\sqrt{2} \right)-\left(3+2\sqrt{2} \right) \lambda, \, 3 \lambda - 1 \right\} \\
&=\left\{ \begin{array}{ll} {\sqrt{2} \over 2} \left[ \left(1+2\sqrt{2} \right)-\left(3+2\sqrt{2} \right) \lambda \right] & \text{if }{1 \over 3} \leq \lambda <{2 \sqrt{2}+1 \over 7}, \\
 {\sqrt{2} \over 2}(3 \lambda-1) & \text{if }{2 \sqrt{2}+1 \over 7} \leq \lambda < 4 \sqrt{2}-5.
\end{array} \right.
\end{aligned}
$$
\item Assume finally $4\sqrt{2}-5 \leq \lambda \leq 1$. Then,
$$
\sup_{-1 \leq t \leq 1}|f_{\lambda}\left(t,{\pi \over 4}\right)|={\sqrt{2} \over 2} \max \left\{3\lambda-1, \, \left(3+2\sqrt{2} \right)\lambda-\left(1+2\sqrt{2} \right) \right\}={\sqrt{2} \over 2}(3\lambda-1).
$$
\end{itemize}
Hence, we can say that
$$
\begin{array}{lcl}
\sup_{-1 \leq t \leq 1}|f_{\lambda}\left(t,{\pi \over 4}\right)| & =  & \left\{\begin{array}{ll}{\sqrt{2} \over 2}\left[1+2\sqrt{2}-\left(3+2\sqrt{2}\right)\lambda \right] & \text{if }0 \leq \lambda<{2 \sqrt{2} +1 \over 7} \vspace{0.2cm} \\
 {\sqrt{2} \over 2} \left(3\lambda-1 \right) & \text{if }{2 \sqrt{2} +1 \over 7} \leq \lambda \leq 1.
\end{array} \right. \vspace{0.2cm} \\
                                                     & =: &\left\{\begin{array}{ll}D_{3,1}(\lambda) & \text{if }0 \leq \lambda<{2 \sqrt{2} +1 \over 7} \vspace{0.2cm} \\
D_{3,2}(\lambda) & \text{if }{2 \sqrt{2} +1 \over 7} \leq \lambda \leq 1.
\end{array} \right.
\end{array}
$$

\medskip

(4) $t=-1$, $0 \leq \theta \leq {\pi \over 4}$.

\medskip

Applying lemma \ref{lem_aux}, we obtain
$$
\begin{array}{lcl}
\sup\limits_{0\leq \theta \leq {\pi\over 4}}f_{\lambda}(-1,\theta) & = & \left \{ \begin{array}{ll}1 & \text{if }0 \leq \lambda <{1+\sqrt{2} \over 3}, \vspace{0.2cm} \\
{\sqrt{2} \over 2}(3\lambda-1) & \text{if }{1+\sqrt{2} \over 3} \leq \lambda \leq 1.
\end{array} \right. \vspace{0.2cm} \\
 & =: & \left \{ \begin{array}{ll}D_{4,1}(\lambda) & \text{if }0 \leq \lambda <{1+\sqrt{2} \over 3}, \vspace{0.2cm} \\
D_{4,2}(\lambda) & \text{if }{1+\sqrt{2} \over 3} \leq \lambda \leq 1.
\end{array} \right.
\end{array}
$$

\medskip

(5) $t=1, \, 0 \leq \theta \leq {\pi \over 4}$.

\medskip

We use again lemma \ref{lem_aux}, with $a=1-\left(2+2\sqrt{2} \right)\lambda$ and $b=\left(5+4\sqrt{2}\right) \lambda-\left(2+2\sqrt{2} \right)$. Through standard calculations, we see that $\frac{b}{a}<0$ if and only if $\lambda\in \left[0,\frac{\sqrt{2}-1}{2}\right)\cup\left(\frac{6-2\sqrt{2}}{7},1\right]$ and $\frac{b}{a}>1$ if and only if $\frac{\sqrt{2}-1}{2}<\lambda<\frac{3+4\sqrt{2}}{23}$. Therefore,
$$
\begin{array}{l}
\sup\limits_{0 \leq \theta \leq {\pi \over 4}}| f_{\lambda}(1, \theta)| \vspace{0.2cm} \\
=\left\{
\begin{array}{ll}
\max \left\{ \left\vert 1-\left( 2+2\sqrt{2}\right) \lambda \right\vert ,%
\frac{\sqrt{2}}{2}\left\vert \left( 3+2\sqrt{2}\right) \lambda -\left( 1+2%
\sqrt{2}\right) \right\vert \right\} & \text{ if } 0\leq \lambda < \frac{3+4\sqrt{2}}{23},  \\
\sqrt{\left( 1-\left( 2+2\sqrt{2}\right) \lambda \right) ^{2}+\left( \left(
5+4\sqrt{2}\right) \lambda -\left( 2+2\sqrt{2}\right) \right) ^{2}} & \text{ if
} \frac{3+4\sqrt{2}}{23}\leq\lambda<\frac{6-2\sqrt{2}}{7},  \\
\max \left\{ \left\vert 1-\left( 2+2\sqrt{2}\right) \lambda \right\vert ,
\frac{\sqrt{2}}{2}\left\vert \left( 3+2\sqrt{2}\right) \lambda -\left( 1+2%
\sqrt{2}\right) \right\vert \right\} & \text{ if } \frac{6-2\sqrt{2}}{7}\leq\lambda\leq 1.
\end{array}
\right.
\end{array}
$$
Since $0\leq\lambda<\sqrt{2}-1$ implies $\left\vert 1-\left( 2+2\sqrt{2}\right) \lambda \right\vert <
\frac{\sqrt{2}}{2}\left\vert \left( 3+2\sqrt{2}\right) \lambda -\left( 1+2%
\sqrt{2}\right) \right\vert$, it follows that
$$
\begin{array}{l}
\sup\limits_{0 \leq \theta \leq {\pi \over 4}}| f_{\lambda}(1,\theta)| \vspace{0.2cm} \\
=\left\{
\begin{array}{ll}
\frac{\sqrt{2}}{2}\left\vert \left( 3+2\sqrt{2}\right) \lambda -\left( 1+2%
\sqrt{2}\right) \right\vert & \text{ if } 0\leq\lambda<\frac{3+4\sqrt{2}%
}{23} \\
\sqrt{48\sqrt{2}\lambda^{2}-56\lambda+69\lambda^{2}-40\sqrt{2}\lambda
+8\sqrt{2}+13} & \text{ if
}\frac{3+4\sqrt{2}}{23}\leq\lambda<\frac{6-2\sqrt{2}}{7} \\
\left\vert 1-\left( 2+2\sqrt{2}\right) \lambda \right\vert & \text{ if }%
\frac{6-2\sqrt{2}}{7}\leq\lambda\leq 1
\end{array}
\right. \vspace{0.2cm} \\
=\left\{
\begin{array}{ll}
\frac{\sqrt{2}}{2}\left[  1+2\sqrt{2} -\left( 3+2\sqrt{2}%
\right) \lambda \right] & \text{ if } 0\leq\lambda<\frac{3+4\sqrt{2}}{23%
}  \\
\sqrt{48\sqrt{2}\lambda^{2}-56\lambda+69\lambda^{2}-40\sqrt{2}\lambda
+8\sqrt{2}+13} & \text{ if
} \frac{3+4\sqrt{2}}{23}\leq\lambda<\frac{6-2\sqrt{2}}{7} \\
\left( 2+2\sqrt{2}\right) \lambda -1 & \text{ if } \frac{6-2%
\sqrt{2}}{7}\leq\lambda\leq 1.
\end{array}%
\right. \vspace{0.2cm} \\
=:\left\{
\begin{array}{ll}
D_{5,1}(\lambda) & \text{ if } 0\leq\lambda<\frac{3+4\sqrt{2}}{23%
}  \\
D_{5,2}(\lambda) & \text{ if
} \frac{3+4\sqrt{2}}{23}\leq\lambda<\frac{6-2\sqrt{2}}{7} \\
D_{5,3}(\lambda) & \text{ if } \frac{6-2%
\sqrt{2}}{7}\leq\lambda\leq 1.
\end{array}%
\right.
\end{array}
$$

\medskip

Since (see Figures \ref{graph5} and \ref{graph6})
$$
\begin{array}{l}
D_{1,1}(\lambda)\leq\left\{
\begin{array}{ll}
D_{2,1}(\lambda) & \text{if }0\leq\lambda<2-\sqrt{2}, \vspace{0.2cm} \\
D_{2,2}(\lambda) & \text{if }2-\sqrt{2}\leq\lambda\leq 1,
\end{array}\right. \vspace{0.2cm} \\
D_{1,2}(\lambda)\leq D_{3,1}(\lambda) \text{ for }0<\lambda<\frac{1}{5},
\end{array}
$$
we can rule out case (1). Since
$$
\begin{array}{ll}
D_{3,1}(\lambda)=D_{5,1}(\lambda) & \text{for }0\leq\lambda\leq\frac{3+4\sqrt{2}}{23},\vspace{0.2cm}\\
D_{3,2}(\lambda)=D_{4,2}(\lambda) & \text{for }\frac{1+\sqrt{2}}{3}\leq\lambda\leq1,
\end{array}
$$
we can directly rule out case (3). Since (see Figures \ref{graph5} and \ref{graph7})
$$
\begin{array}{l}
D_{4,1}(\lambda)=1\leq\left\{
\begin{array}{ll}
D_{2,1}(\lambda) & \text{if }0\leq\lambda<2-\sqrt{2}, \vspace{0.2cm} \\
D_{2,2}(\lambda) & \text{if }2-\sqrt{2}\leq\lambda<\frac{1+\sqrt{2}}{3},
\end{array}\right. \vspace{0.2cm} \\
D_{4,2}(\lambda)\leq D_{2,2} \text{ for }\frac{1+\sqrt{2}}{3}\leq\lambda\leq1,
\end{array}
$$
we can rule out case (4). Finally, since (see Figure \ref{graph8})
$$
\begin{array}{ll}
D_{5,2}(\lambda)\leq D_{2,1}(\lambda) & \text{for }\frac{3+4\sqrt{2}}{23}\leq\lambda<\frac{6-2\sqrt{2}}{7}, \vspace{0.2cm} \\
D_{5,3}(\lambda)=D_{2,2}(\lambda) & \text{for }2-\sqrt{2}\leq\lambda\leq 1,
\end{array}
$$
we can rule out the expressions $D_{5,2}(\lambda)$ and $D_{5,3}(\lambda)$ of case (5).

Thus, putting all the above cases together, we may reach the conclusion
\begin{eqnarray*}
\sup\limits_{(t,\theta) \in C_1}|f_{\lambda}(t,\theta)| & = & \left\{\begin{array}{ll} D_{5,1}(\lambda) & \text{if }0 \leq \lambda < {(2-3\sqrt{2})\sqrt{4\sqrt{2}+7} +5\sqrt{2} +6 \over 14}, \\
D_{2,1}(\lambda) & \text{if }{(2-3\sqrt{2})\sqrt{4\sqrt{2}+7} +5\sqrt{2} +6 \over 14} \leq \lambda < 2-\sqrt{2}, \\
D_{2,2}(\lambda) & \text{if }2-\sqrt{2} \leq \lambda \leq 1,
\end{array} \right. \\
                                                        & = & \left\{\begin{array}{ll} {\sqrt{2} \over 2}\left[ \left(1+2\sqrt{2}\right)-\left(3+2\sqrt{2}\right)\lambda\right] & \text{if }0 \leq \lambda < {(2-3\sqrt{2})\sqrt{4\sqrt{2}+7} +5\sqrt{2} +6 \over 14}, \\
1+{\lambda^2 \over 1-\lambda} & \text{if }{(2-3\sqrt{2})\sqrt{4\sqrt{2}+7} +5\sqrt{2} +6 \over 14} \leq \lambda < 2-\sqrt{2}, \\
\left(2+2\sqrt{2}\right) \lambda-1 & \text{if }2-\sqrt{2} \leq \lambda \leq 1,
\end{array} \right.
\end{eqnarray*}
and hence
$$
\begin{array}{l}
\sup\limits_{-1 \leq t \leq 1} \| D P_t(x,y) \|_{D({\pi \over 4})}=2x \sup\limits_{(t,\theta)\in C_1} |f_{\lambda}(t, \theta)| \\
=\left\{
\begin{array}{ll}
\sqrt{2}\left[\left(1+2\sqrt{2}\right)x-\left(3+2\sqrt{2}\right)y\right] & \text{if }0 \leq y < {(2-3\sqrt{2})\sqrt{4\sqrt{2}+7} +5\sqrt{2} +6 \over 14}x, \\
2\left(x+{y^2 \over x-y}\right) & \text{if }{(2-3\sqrt{2})\sqrt{4\sqrt{2}+7} +5\sqrt{2} +6 \over 14}x \leq y < \left(2-\sqrt{2}\right)x, \\
4\left(1+\sqrt{2}\right)y-2x  & \text{if }\left(2-\sqrt{2}\right)x \leq y \leq x, \end{array} \right.
\end{array}
$$
assuming in every moment $x \neq 0$ (in order to illustrate the previous step, the reader can take a look at Figure \ref{graph9}).

\begin{figure}
\centering
\includegraphics[width=0.9\textwidth]{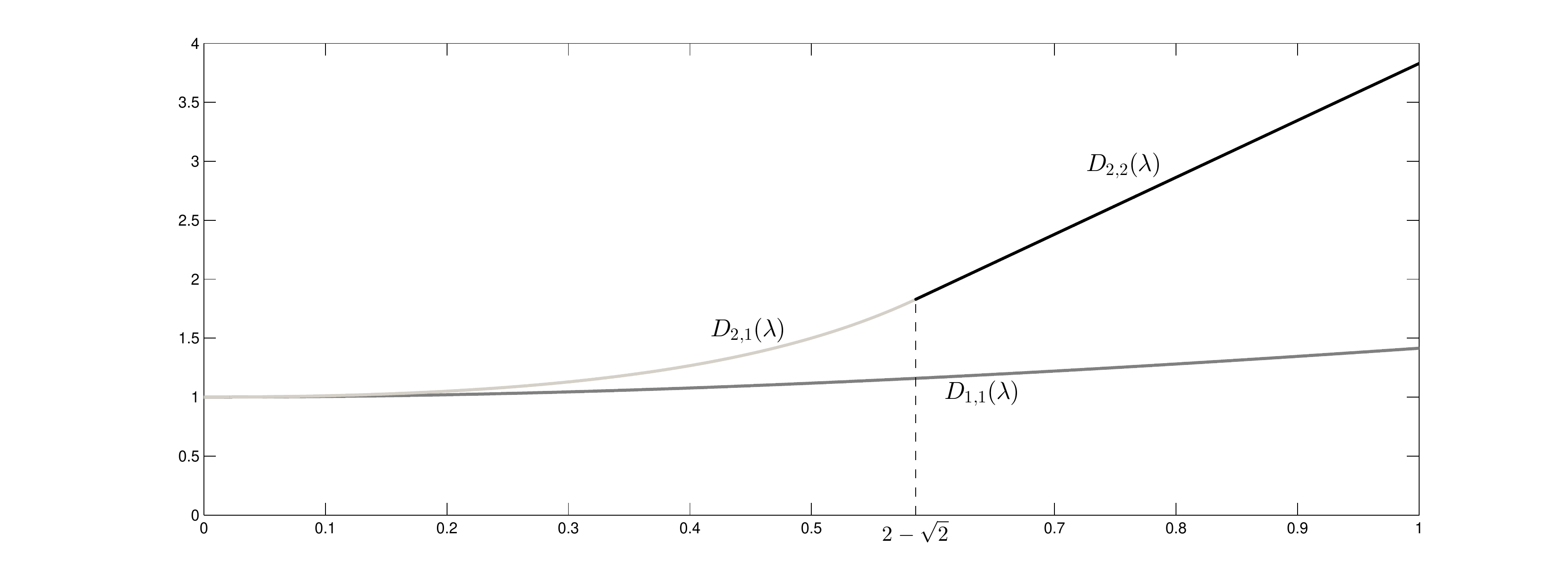}
\caption{Graphs of the mappings $D_{1,1}(\lambda)$, $D_{2,1}(\lambda)$ and $D_{2,2}(\lambda)$.}\label{graph5}
\end{figure}

\begin{figure}
\centering
\includegraphics[width=0.9\textwidth]{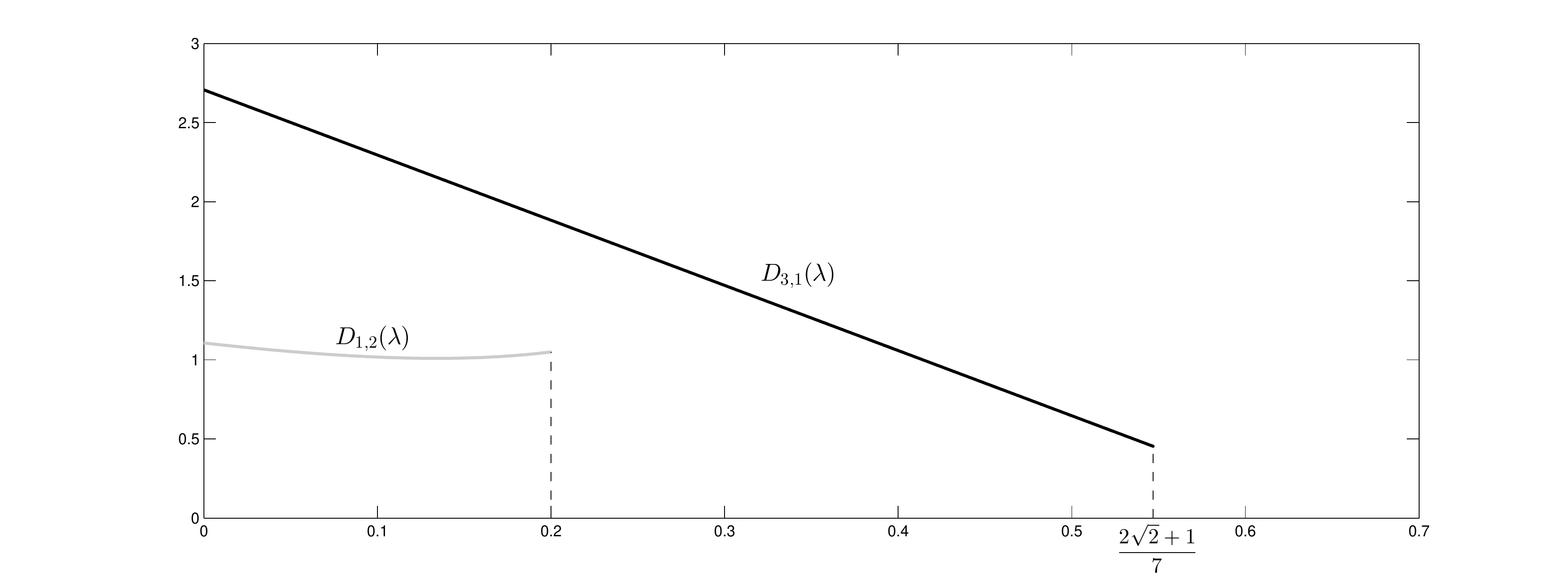}
\caption{Graphs of the mappings $D_{1,2}(\lambda)$ and $D_{3,1}(\lambda)$.}\label{graph6}
\end{figure}

\begin{figure}
\centering
\includegraphics[width=0.9\textwidth]{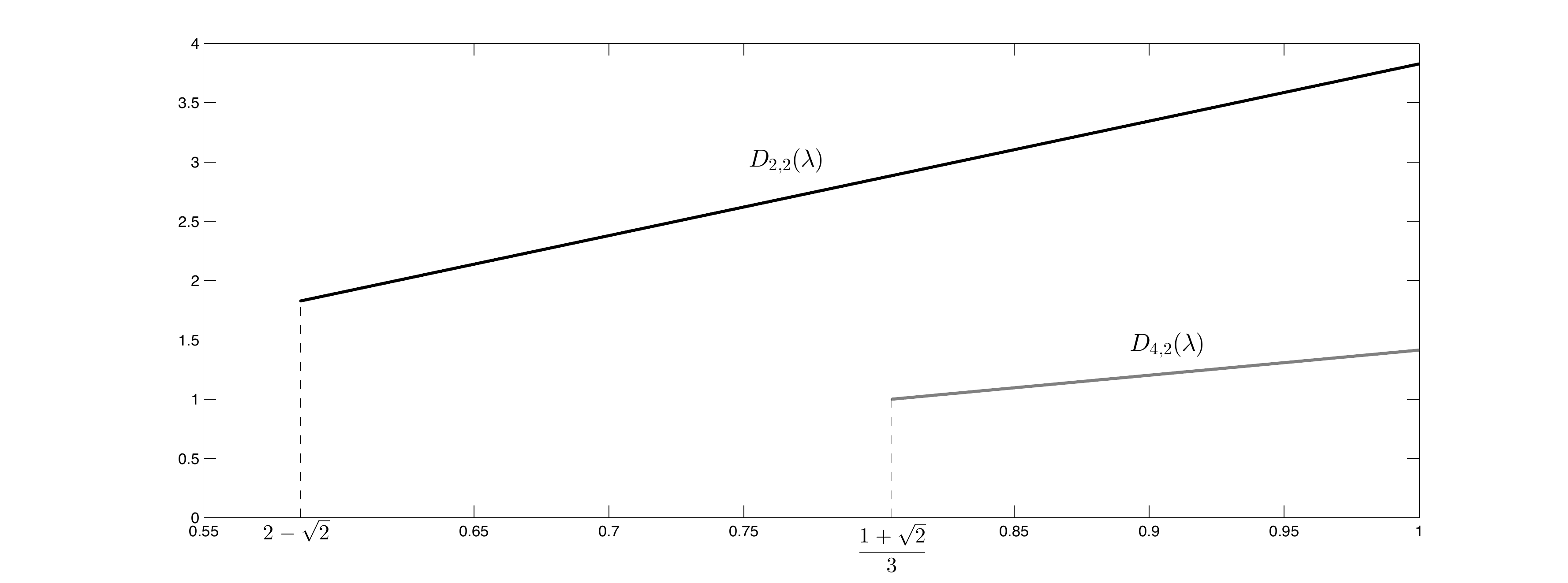}
\caption{Graphs of the mappings $D_{2,2}(\lambda)$ and $D_{4,2}(\lambda)$.}\label{graph7}
\end{figure}

\begin{figure}
\centering
\includegraphics[width=0.9\textwidth]{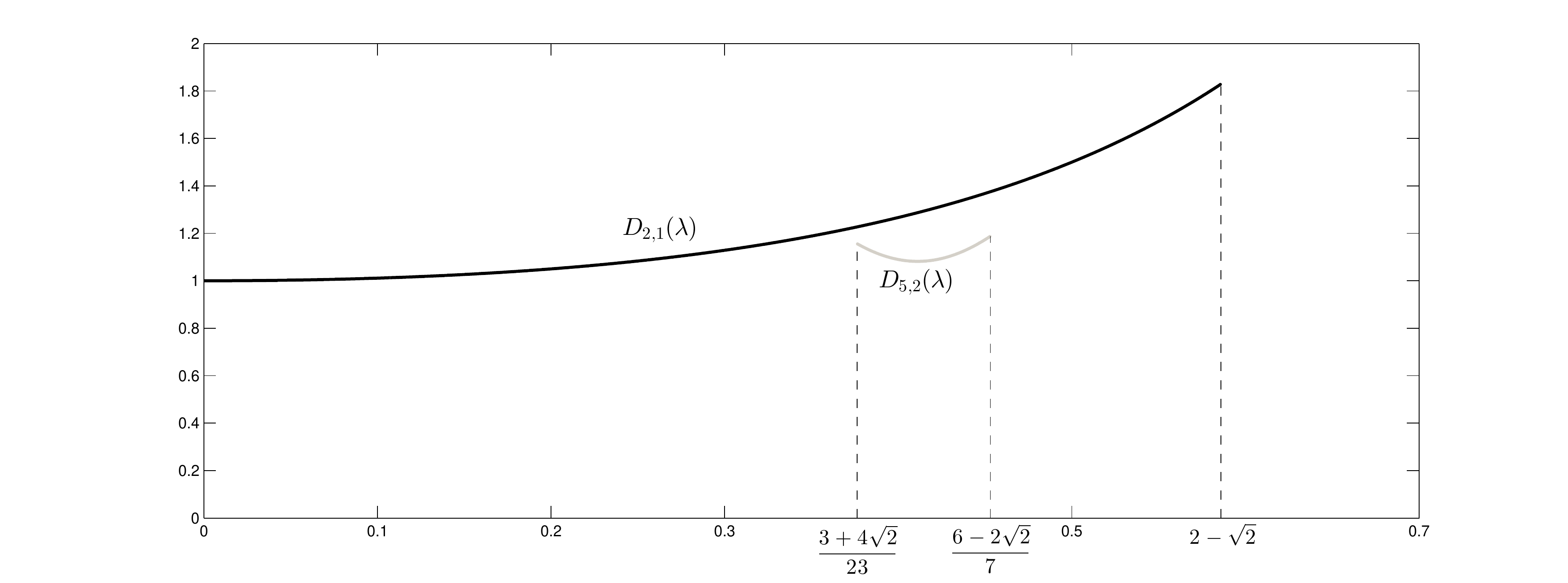}
\caption{Graphs of the mappings $D_{2,1}(\lambda)$ and $D_{5,2}(\lambda)$.}\label{graph8}
\end{figure}

\begin{figure}
\centering
\includegraphics[width=0.9\textwidth]{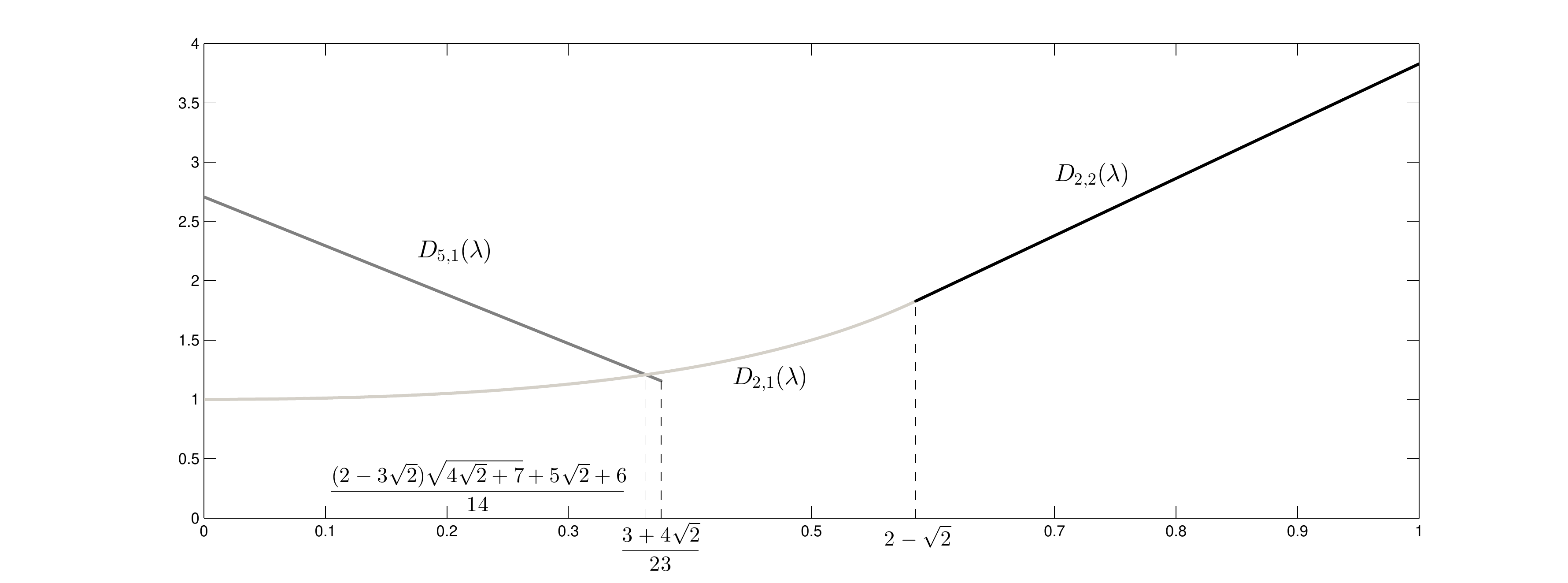}
\caption{Graphs of the mappings $D_{2,1}(\lambda)$, $D_{2,2}(\lambda)$ and $D_{5,1}(\lambda)$.}\label{graph9}
\end{figure}

\medskip

Let us deal now with the polynomials
$$
Q_s(x,y)=x^2+sy^2-2\sqrt{2(1+s)}xy, \quad 1 \leq s \leq 5+4\sqrt{2}.
$$
Then,
$$
\begin{aligned}
&\nabla Q_s(x,y)=\left(2x-2\sqrt{2(1+s)}y, \, 2sy-2\sqrt{2(1+s)}x\right), \\
&\|D Q_s(x,y) \|_{D({\pi \over 4})}=\sup_{0 \leq \theta \leq {\pi \over 4}} \left|2x\left[\left(1-\sqrt{2(1+s)}\lambda\right)\cos \theta+\left(s\lambda-\sqrt{2(1+s)}\right) \sin \theta \right] \right|,
\end{aligned}
$$
and thus
$$
\sup_{1 \leq s \leq 5+4\sqrt{2}}\|D Q_s(x,y) \|_{D({\pi \over 4})} =2x \sup_{(s,\theta) \in C_2}|g_{\lambda}(s,\theta)|,
$$
with
$$g_{\lambda}(s, \theta)=\left(1-\sqrt{2(1+s)}\lambda \right) \cos \theta+\left(s\lambda-\sqrt{2(1+s)}\right) \sin \theta
$$
and $C_2=[1,5+4\sqrt{2}] \times [0,{\pi \over 4}]$. Again, we have several cases:

\medskip

(6) $(s,\theta)\in (1,5+4\sqrt{2}) \times (0,{\pi \over 4})$.

\medskip

Let us first calculate the critical points of $g_{\lambda}$ over $C_2$.
$$
\begin{aligned}
& {\partial g_{\lambda} \over \partial s}(s_0,\theta_0)={-\lambda \over \sqrt{2(1+s_0)}} \cos \theta_0+\left( \lambda-{1 \over \sqrt{2(1+s_0)}} \right) \sin \theta_0, \\
& {\partial g_{\lambda} \over \partial \theta}(s_0,\theta_0)=\left(s_0 \lambda-\sqrt{2(1+s_0)}\right) \cos \theta_0-\left(1-\sqrt{2(1+s_0)}\lambda \right) \sin \theta_0,
\end{aligned}
$$
so, if $D g_{\lambda}(s_0,\theta_0)=0$, using the first expression, we obtain $\tan \theta_0={\lambda \over \sqrt{2(1+s_0)}\lambda-1}$, and, using the second one, we obtain $\tan \theta_0={s_0\lambda-\sqrt{2(1+s_0)} \over 1-\sqrt{2(1+s_0)}\lambda}$. \\
\noindent
Hence, we may say
$$
{s_0 \lambda-\sqrt{2(1+s_0)} \over 1-\sqrt{2(1+s_0)} \lambda}={\lambda \over \sqrt{2(1+s_0)}\lambda-1}
$$
and thus
$$
s_0={2-\lambda^2 \over \lambda^2}.
$$
Then, $\tan \theta_0=\lambda$ and also, if we want to guarantee that $1<s_0<5+4\sqrt{2}$, we need $\sqrt{2}-1< \lambda <1$. \\
\noindent
In that case, $\sin \theta_0={\lambda \over \sqrt{1+\lambda^2}}$ and $\cos \theta_0={1 \over \sqrt{1+\lambda^2}}$, and then
$$
g_{\lambda}(s_0, \theta_0)={-1 \over \sqrt{1+\lambda^2}}+{-\lambda^2 \over \sqrt{1+\lambda^2}}=-\sqrt{1+\lambda^2},
$$
so $$D_{6}(\lambda):=|g_{\lambda}(s_0,\theta_0)|=\sqrt{1+\lambda^2}.$$

\medskip

(7) $s=1, \, 0 \leq \theta \leq {\pi \over 4}$.

\medskip

Apply lemma \ref{lem_aux} with $a=1-2\lambda$ and $b=\lambda-2$. Using $0\leq \lambda \leq 1$, observe that we always have $b < 0$ and $b \leq a$. Also, $a<\left(1-\sqrt{2}\right)b$ if and only if $\lambda>{5-3\sqrt{2} \over 7}$. \\
\noindent
Putting everything together, we can say
\begin{eqnarray*}
\sup_{0 \leq \theta \leq {\pi \over 4}}|g_{\lambda}(1,\theta)| & =  &\begin{cases}
1-2\lambda & \text{if }0 \leq \lambda < {5-3\sqrt{2} \over 7}, \\
{\sqrt{2} \over 2}(1+\lambda) & \text{if }{5-3\sqrt{2} \over 7} \leq \lambda \leq 1,
\end{cases} \\
                                                               & =: & \begin{cases}
D_{7,1}(\lambda) & \text{if }0 \leq \lambda < {5-3\sqrt{2} \over 7}, \\
D_{7,2}(\lambda) & \text{if }{5-3\sqrt{2} \over 7} \leq \lambda \leq 1.
\end{cases}
\end{eqnarray*}

\medskip

(8) $s=5+4\sqrt{2}, \, 0 \leq \theta \leq {\pi \over 4}$.

\medskip

Apply again lemma \ref{lem_aux}, this time to $a=1-2\left( 1+\sqrt{2}\right)\lambda$ and $b=\left(5+4\sqrt{2}\right)\lambda-2\left(1+\sqrt{2}\right)$. As usual, we notice that $a<0$ if and only if $\lambda>\frac{\sqrt{2}-1}{2}$, $b<0$ if and only if $\lambda<\frac{6-2\sqrt{2}}{7}$ and $a<b$ if and only if $\lambda>\frac{3+4\sqrt{2}}{23}$. All together, we can say that, for ${3+4\sqrt{2} \over 23}< \lambda < {6-2\sqrt{2} \over 7}$, we have
$$
\sup_{0 \leq \theta \leq {\pi \over 4}}|g_{\lambda}(5+4\sqrt{2},\theta)| =\sqrt{a^2+b^2}=\sqrt{13+8\sqrt{2}-\left(56+40\sqrt{2}\right)\lambda+\left(69+48\sqrt{2}\right)\lambda^2}.
$$
Also, notice that, for any $\lambda \in [0,1],$ we are going to have $b<-\left(1+\sqrt{2}\right)a$ and $a<\left(1-\sqrt{2}\right)b$. Hence,
$$
\begin{array}{l}
\sup\limits_{0 \leq \theta \leq {\pi \over 4}}|g_{\lambda}(5+4\sqrt{2},\theta)| \vspace{0.2cm} \\
=\begin{cases}{\sqrt{2} \over 2}\left[\left(1+2\sqrt{2} \right) -\left(3+2\sqrt{2} \right)\lambda\right] & \text{if }0 \leq \lambda <{3+4\sqrt{2} \over 23}, \\
\sqrt{13+8\sqrt{2}-\left(56+40\sqrt{2}\right)\lambda+\left(69+48\sqrt{2}\right)\lambda^2} & \text{if }{3+4\sqrt{2} \over 23}\leq \lambda < {6-2\sqrt{2} \over 7},\\
2\left(1+\sqrt{2}\right)\lambda-1 & \text{if }{6-2\sqrt{2} \over 7} \leq \lambda \leq 1,
\end{cases} \vspace{0.2cm} \\
=:\begin{cases}
D_{8,1}(\lambda) & \text{if }0 \leq \lambda <{3+4\sqrt{2} \over 23}, \\
D_{8,2}(\lambda) & \text{if }{3+4\sqrt{2} \over 23}\leq \lambda < {6-2\sqrt{2} \over 7},\\
D_{8,3}(\lambda) & \text{if }{6-2\sqrt{2} \over 7} \leq \lambda \leq 1. \end{cases}
\end{array}
$$

\medskip

(9) $\theta=0, \, 1 \leq s \leq 5+4\sqrt{2}$.

\medskip

We have
$$
\begin{aligned}
&g_{\lambda}(s,0)=1-\sqrt{2(1+s)}\lambda, \\
& g_{\lambda}(1,0)=1-2\lambda, \\
& g_{\lambda}(5+4\sqrt{2},0)=1-2\left(1+\sqrt{2}\right)\lambda, \\
& g_{\lambda}'(s,0)=-{\lambda \over \sqrt{2(1+s)}} \neq 0 \text{ for }\lambda \neq 0.
\end{aligned}
$$
Then,
$$
\begin{aligned}
\sup_{1 \leq s \leq 5+4\sqrt{2}}|g_{\lambda}(s,0)|&=\max\left\{|1-2\lambda|, \, |1-2(1+\sqrt{2})\lambda| \right\} \\
&=\left\{ \begin{array}{ll}
1-2\lambda & \text{if }0 \leq \lambda < {2-\sqrt{2} \over 2}, \\
2\left(1+\sqrt{2}\right)\lambda-1 & \text{if }{2-\sqrt{2} \over 2} \leq \lambda \leq 1,
\end{array} \right. \\
&=:\left\{ \begin{array}{ll}
D_{9,1}(\lambda) & \text{if }0 \leq \lambda < {2-\sqrt{2} \over 2}, \\
D_{9,2}(\lambda) & \text{if }{2-\sqrt{2} \over 2} \leq \lambda \leq 1.
\end{array} \right.
\end{aligned}
$$

\medskip

(10) $\theta=\frac{\pi}{4}, \, 1 \leq s \leq 5+4\sqrt{2}$.

\medskip

We have $$g_{\lambda}\left(s, {\pi \over 4}\right)={\sqrt{2} \over 2}\left[1+s\lambda-\sqrt{2(1+s)}(1+\lambda)\right].$$ Then
$$
\begin{aligned}
& g_{\lambda}\left(1,{\pi \over 4}\right)=-\frac{\sqrt{2}}{2}(1+\lambda), \\
& g_{\lambda}\left(5+4\sqrt{2},{\pi \over 4}\right)={\sqrt{2} \over 2} \left[\left(3+2\sqrt{2} \right) \lambda-\left(1+2\sqrt{2}\right) \right], \\
& g_{\lambda}'\left(s_0,{\pi \over 4}\right)=0 \text{ if and only if }s_0={(1+\lambda)^2 \over 2\lambda^2}-1
\end{aligned}
$$
and since we need to ensure that $1 < s_0 < 5+4\sqrt{2}$, we need ${2\sqrt{2}-1 \over 7}<\lambda<1$. In that case,
$$
g_{\lambda}\left(s_0,{\pi \over 4}\right)=-{\sqrt{2}(1+3\lambda^2) \over 4\lambda}.
$$
Hence,
\begin{eqnarray*}
\sup_{1 \leq s \leq 5+4\sqrt{2}} \left|g_{\lambda}\left(s.{\pi \over 4}\right)\right| & =  & \left\{ \begin{array}{ll}
{\sqrt{2} \over 2}\left[\left(1+2\sqrt{2} \right) -\left(3+2\sqrt{2} \right)\lambda\right] & \text{if }0 \leq \lambda < \frac{2\sqrt{2} -1}{7}, \vspace{0.2cm} \\
{\sqrt{2}(1+3\lambda^2) \over 4\lambda} & \text{if }\frac{2\sqrt{2}-1}{7} \leq \lambda \leq 1,
\end{array} \right. \\
 & =:  & \left\{ \begin{array}{ll}
D_{10,1}(\lambda) & \text{if }0 \leq \lambda < \frac{2\sqrt{2} -1}{7}, \vspace{0.2cm} \\
D_{10,2}(\lambda) & \text{if }\frac{2\sqrt{2}-1}{7} \leq \lambda \leq 1.
\end{array} \right.
\end{eqnarray*}

\medskip

Since (the reader can take a look at Figure \ref{graph10})
$$
D_6(\lambda)\leq\left\{
\begin{array}{ll}
D_{8,2}(\lambda) & \text{if } \sqrt{2}-1<\lambda<\frac{6-2\sqrt{2}}{7}, \vspace{0.2cm} \\
D_{8,3}(\lambda) & \text{if } \frac{6-2\sqrt{2}}{7}\leq\lambda<1,
\end{array}\right.
$$
we can rule out case (6). Since (see Figures \ref{graph11} and \ref{graph12})
$$
\begin{array}{l}
D_{7,1}(\lambda)\leq D_{10,1}(\lambda) \text{ for } 0\leq\lambda<\frac{5-3\sqrt{2}}{7} \vspace{0.2cm} \\
D_{7,2}(\lambda)\leq\left\{
\begin{array}{ll}
D_{10,1}(\lambda) & \text{if }\frac{5-3\sqrt{2}}{7}\leq\lambda<\frac{2\sqrt{2}-1}{7}, \vspace{0.2cm} \\
D_{10,2}(\lambda) & \text{if }\frac{2\sqrt{2}-1}{7}\leq\lambda\leq 1,
\end{array}\right.
\end{array}
$$
we can rule out case (7). Since $$D_{8,1}(\lambda)=D_{10,1}(\lambda) \text{ for }0\leq\lambda<\frac{2\sqrt{2}-1}{7}$$
we can rule out the expression $D_{8,1}(\lambda)$ of case (8). Since
$$
\begin{array}{ll}
D_{9,1}(\lambda) = D_{7,1}(\lambda) & \text{for }0\leq\lambda<\frac{5-3\sqrt{2}}{7}, \vspace{0.2cm} \\
D_{9,2}(\lambda) = D_{8,3}(\lambda) & \text{for }\frac{6-2\sqrt{2}}{7}\leq\lambda\leq1,
\end{array}
$$
we can directly rule out case (9). Furthermore, since (see Figure \ref{graph13})
$$
\begin{array}{l}
D_{8,2}(\lambda)\leq D_{10,2}(\lambda) \text{ for }\frac{3+4\sqrt{2}}{23}\leq\lambda< \frac{6-2\sqrt{2}}{7}, \vspace{0.2cm} \\
D_{8,3}(\lambda)\leq D_{10,2}(\lambda) \text{ for }\frac{6-2\sqrt{2}}{7}\leq\lambda\leq\frac{(4\sqrt{2}-5)\sqrt{4\sqrt{2}+7}+8-5\sqrt{2}}{7},
\end{array}
$$
we can conclude that
\begin{eqnarray*}
\sup_{(s, \theta) \in C_2}|g_{\lambda}(s,\theta)| & = & \left\{ \begin{array}{ll}
D_{10,1}(\lambda) & \text{if }0 \leq \lambda<{2\sqrt{2}-1 \over 7}, \\
D_{10,2}(\lambda) & \text{if }{2\sqrt{2}-1 \over 7} \leq \lambda < \frac{(4\sqrt{2}-5)\sqrt{4\sqrt{2}+7}+8-5\sqrt{2}}{7}, \\
D_{8,3}(\lambda) & \text{if }\frac{(4\sqrt{2}-5)\sqrt{4\sqrt{2}+7}+8-5\sqrt{2}}{7} \leq \lambda \leq 1.
\end{array} \right. \\
                                                  & = & \left\{ \begin{array}{ll}
{\sqrt{2} \over 2}\left[1+2\sqrt{2}-\left(3+2\sqrt{2} \right) \lambda \right] & \text{if }0 \leq \lambda<{2\sqrt{2}-1 \over 7}, \\
{\sqrt{2}(1+3\lambda^2) \over 4\lambda} & \text{if }{2\sqrt{2}-1 \over 7} \leq \lambda < \frac{(4\sqrt{2}-5)\sqrt{4\sqrt{2}+7}+8-5\sqrt{2}}{7}, \\
2\left(1+\sqrt{2}\right) \lambda-1 & \text{if }\frac{(4\sqrt{2}-5)\sqrt{4\sqrt{2}+7}+8-5\sqrt{2}}{7} \leq \lambda \leq 1,
\end{array} \right.
\end{eqnarray*}
and hence
$$
\begin{array}{l}
\sup\limits_{1 \leq s \leq 5+4\sqrt{2}}\| D Q_s(x,y)\|_{D({\pi \over 4})} \vspace{0.2cm} \\
=\left\{ \begin{array}{ll}
\sqrt{2}\left[\left(1+2\sqrt{2}\right)x-\left(3+2\sqrt{2}\right)y\right] & \text{if }0 \leq y<{2\sqrt{2}-1 \over 7}x, \\
{\sqrt{2}(x^2+3y^2) \over 2y} & \text{if }{2\sqrt{2}-1 \over 7}x \leq y < \frac{(4\sqrt{2}-5)\sqrt{4\sqrt{2}+7}+8-5\sqrt{2}}{7}x, \\
4\left(1+\sqrt{2}\right)y-2x & \text{if }\frac{(4\sqrt{2}-5)\sqrt{4\sqrt{2}+7}+8-5\sqrt{2}}{7}x \leq y \leq x.
\end{array} \right.
\end{array}
$$

\begin{figure}
\centering
\includegraphics[width=0.9\textwidth]{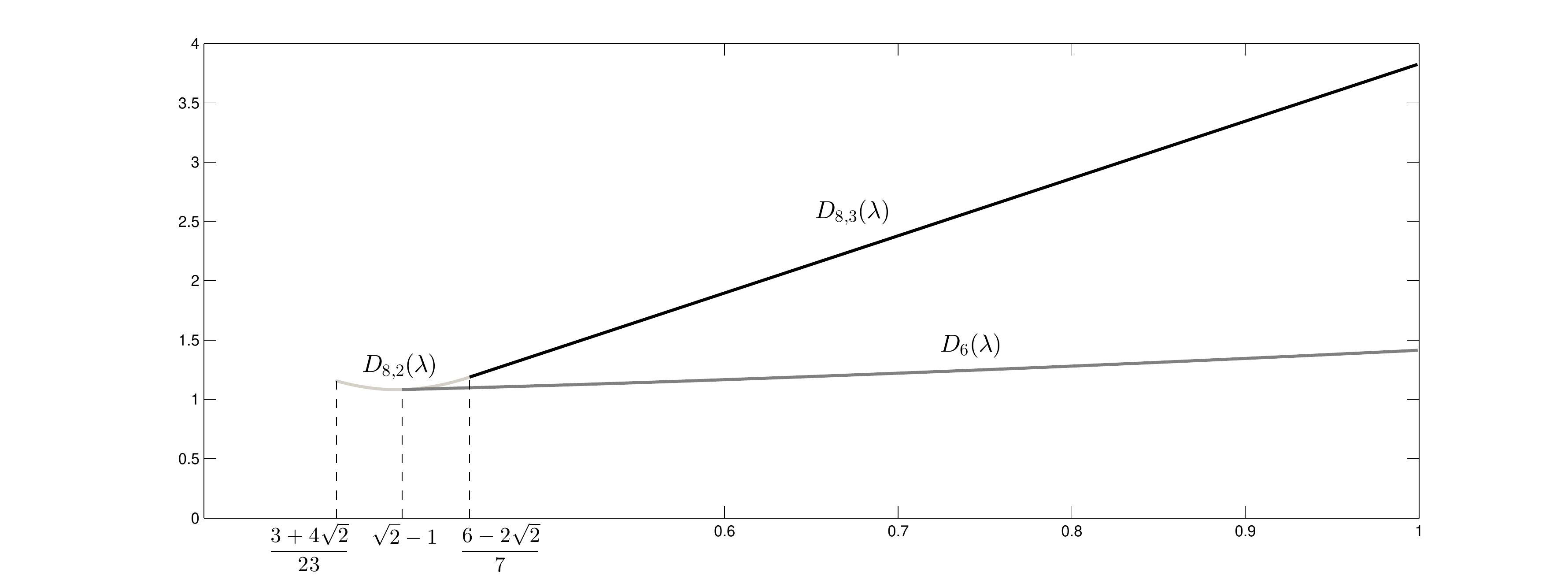}
\caption{Graphs of the mappings $D_{6}(\lambda)$, $D_{8,2}(\lambda)$ and $D_{8,3}(\lambda)$.}\label{graph10}
\end{figure}

\begin{figure}
\centering
\includegraphics[width=0.9\textwidth]{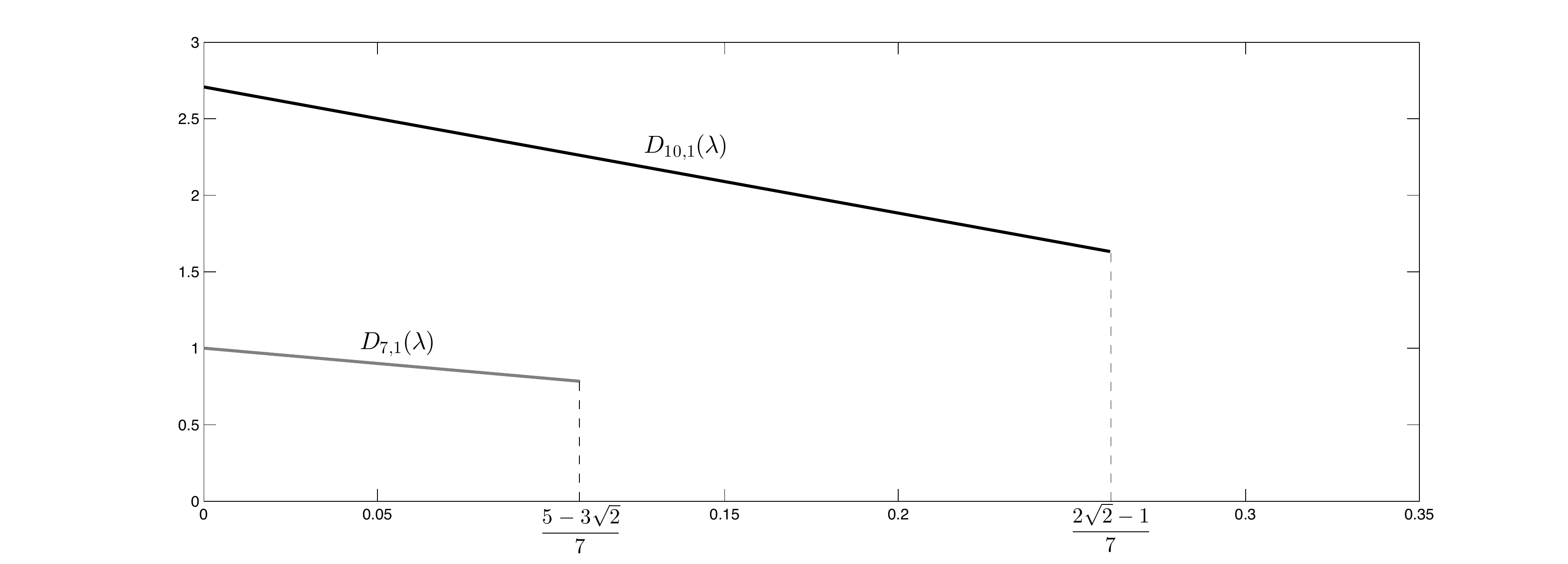}
\caption{Graphs of the mappings $D_{7,1}(\lambda)$ and $D_{10,1}(\lambda)$.}\label{graph11}
\end{figure}

\begin{figure}
\centering
\includegraphics[width=0.9\textwidth]{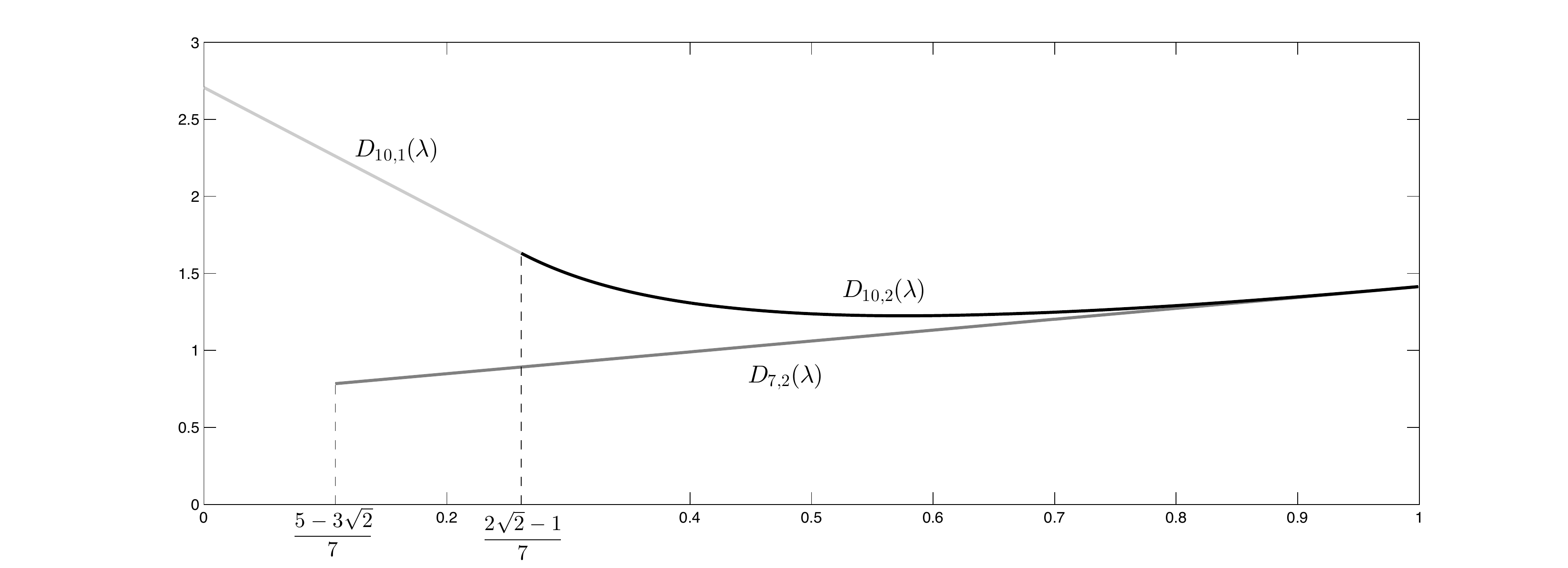}
\caption{Graphs of the mappings $D_{7,2}(\lambda)$, $D_{10,1}(\lambda)$ and $D_{10,2}(\lambda)$.}\label{graph12}
\end{figure}

\begin{figure}
\centering
\includegraphics[width=0.9\textwidth]{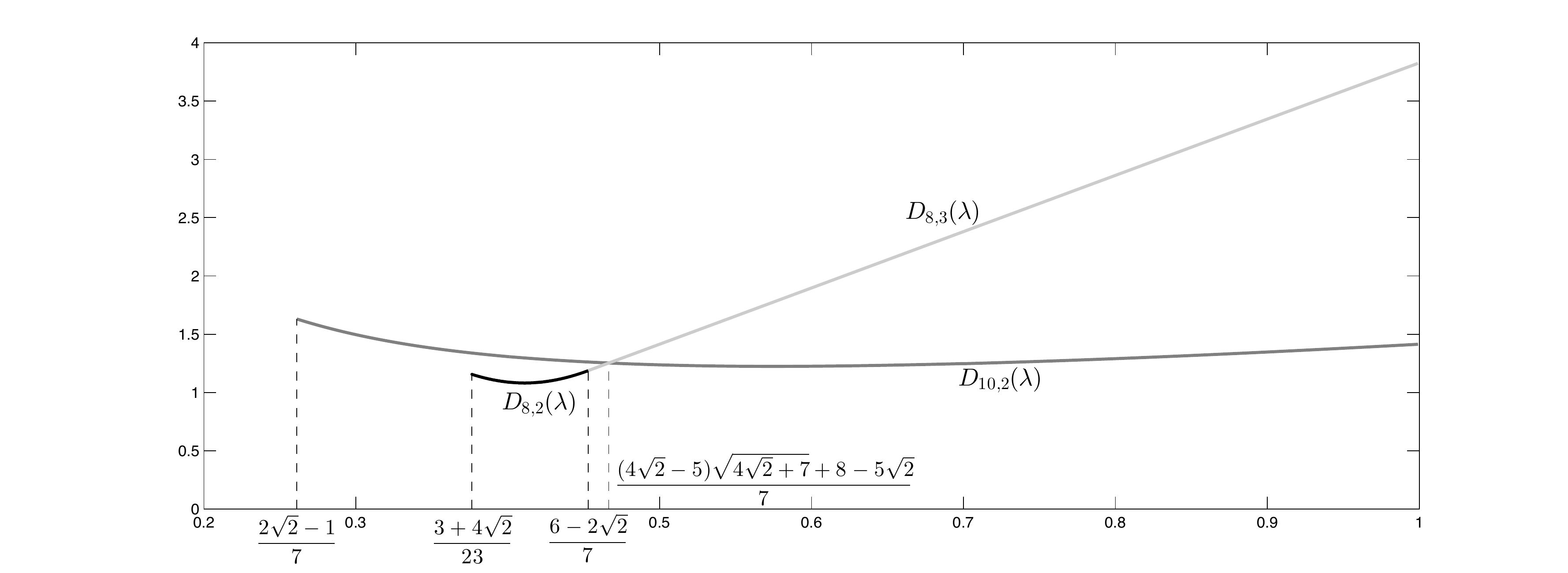}
\caption{Graphs of the mappings $D_{8,2}(\lambda)$, $D_{8,3}(\lambda)$ and $D_{10,2}(\lambda)$.}\label{graph13}
\end{figure}

Finally, if we compare the results obtained with $P_t$ and $Q_s$, since ${\sqrt{2}(1+3\lambda^2) \over 4\lambda}\geq 1+\frac{\lambda^2}{1-\lambda}$ whenever $\lambda\leq\sqrt{2}-1$, we obtain
$$
\Phi(x,y)=\left\{ \begin{array}{ll} \sqrt{2}\left[\left(1+2\sqrt{2}\right)x-\left(3+2\sqrt{2}\right)y\right] & \text{if }0 \leq y < {2\sqrt{2}-1 \over 7}x, \\
{\sqrt{2}(x^2+3y^2) \over 2y} & \text{if }{2\sqrt{2}-1 \over 7}x \leq y < \left(\sqrt{2}-1\right) x, \\
2\left(x+{y^2 \over x-y}\right) & \text{if }\left(\sqrt{2}-1\right) x \leq y < \left(2-\sqrt{2}\right)x, \\
4\left(1+\sqrt{2}\right)y-2x & \text{if }\left(2-\sqrt{2}\right)x \leq y \leq x.
\end{array} \right.
$$
\end{proof}

We can see that $\Phi(x,y) \leq 4+\sqrt{2}$, for all $(x,y)\in D\left(\frac{\pi}{4}\right)$. Furthermore, the maximum is attained by the polynomials
$$
P_1(x,y)=x^2+\left(5+4\sqrt{2} \right) y^2-\left(4+4\sqrt{2}\right)xy=Q_{5+4\sqrt{2}}(x,y).
$$

\begin{corollary}
Let $P\in{\mathcal P}\left(^2D\left(\frac{\pi}{4}\right)\right)$ and assume $L\in{\mathcal L}^s\left(^2D\left(\frac{\pi}{4}\right)\right)$ is the polar of $P$. Then
    $$
    \|L\|_{D\left(\frac{\pi}{4}\right)}\leq \left(2+\frac{\sqrt{2}}{2}\right)\|P\|_{D\left(\frac{\pi}{4}\right)}.
    $$
Moreover, equality is achieved for $P_1(x,y)=Q_{5+4\sqrt{2}}(x,y)=x^2+\left(5+4\sqrt{2} \right) y^2-\left(4+4\sqrt{2}\right)xy$.
Hence, the polarization constant of the polynomial space ${\mathcal P}\left(^2D\left(\frac{\pi}{4}\right)\right)$ is
$2+\frac{\sqrt{2}}{2}$.
\end{corollary}


\section{Unconditional constants for polynomials on sectors}\label{sec:unc_const}

Here, we obtain a sharp estimate on the norm of the modulus of a polynomial in ${\mathcal P}\left(^2D\left(\frac{\pi}{4}\right)\right)$ in terms of it norm. That sharp estimate turns out to be the unconditional constant of the canonical basis of ${\mathcal P}\left(^2D\left(\frac{\pi}{4}\right)\right)$.

\begin{theorem}
The unconditional constant of the canonical basis of ${\mathcal P}\left(^2D\left(\frac{\pi}{4}\right)\right)$ is $5+4\sqrt{2}$. In other words, the inequality
    $$
    \||P|\|_{D\left(\frac{\pi}{4}\right)}\leq (5+4\sqrt{2})\|P\|_{D\left(\frac{\pi}{4}\right)},
    $$
for all $P\in {\mathcal P}\left(^2D\left(\frac{\pi}{4}\right)\right)$. Furthermore, the previous inequality is sharp and equality is attained for the polynomials $\pm P_1(x,y)=\pm Q_{5+4\sqrt{2}}(x,y)=\pm\left[x^2+(5+4\sqrt{2})y^2-(4+4\sqrt{2})xy\right]$.
\end{theorem}

\begin{proof}
We just need to calculate
    $$
    \sup\left\{\||P|\|_{D\left(\frac{\pi}{4}\right)}:\ P\in\ext\left(B_{D\left({\pi \over 4}\right)}\right)\right\}.
    $$
In order to calculate the above supremum we use the extreme polynomials described in Lemma \ref{lem_ext}. If we consider first the polynomials $P_t$, then $|P_t|=\left(|t|,4+t+4\sqrt{1+t},2+2t+4\sqrt{1+t}\right)$. Now, using Lemma \ref{lem_norm} we have
    \begin{align*}
     \sup_{-1\leq t\leq 1}\||P_t|\|_{D\left(\frac{\pi}{4}\right)}&=\sup_{-1\leq t\leq 1}\max\left\{|t|,\frac{1}{2}\left(|t|+4+t+4\sqrt{1+t}+2+2t+4\sqrt{1+t}\right)\right\}\\
    &=\sup_{-1\leq t\leq 1}\frac{1}{2}\left(|t|+6+3t+8\sqrt{1+t}\right)=5+4\sqrt{2}.
    \end{align*}
Notice that the above supremum is attained at $t=1$.
On the other hand, if we consider the polynomials $Q_s$, we have $|Q_s|=\left(1,s,2\sqrt{2(1+s)}\right)$. Now, using Lemma \ref{lem_norm} we have
    \begin{align*}
    \sup_{1\leq s\leq 5+4\sqrt{2}}\||Q_s|\|_{D\left(\frac{\pi}{4}\right)}&=\sup_{1\leq s\leq 5+4\sqrt{2}}\max\left\{1,\frac{1}{2}\left(1+s+2\sqrt{2(1+s)}\right)\right\}\\
    &=\sup_{1\leq s\leq 5+4\sqrt{2}}\frac{1}{2}\left(1+s+2\sqrt{2(1+s)}\right)=5+4\sqrt{2}.
    \end{align*}
Observe that the last supremum is now attained at $s=5+4\sqrt{2}$.
\end{proof}
\section{Conclusions}

Comparing the results obtained in \cite{GMS2} and \cite{MRS} for polynomials on the simplex $\Delta$, in \cite{GMSS} for polynomials on the unit square $\Box$, in \cite{JMPS} for polynomials on the sector $D\left(\frac{\pi}{2}\right)$ and the results obtained in the previous sections, we have the following:

\begin{center}
\bigskip
\renewcommand{\arraystretch}{1.5}
\begin{tabular}{|c||c|c|c|c|}
  \hline
   & ${\mathcal P}(^2\Delta)$ & ${\mathcal P}\left(^2D\left(\frac{\pi}{2}\right)\right)$ & ${\mathcal P}\left(^2D\left(\frac{\pi}{4}\right)\right)$ & ${\mathcal P}(^2\Box)$  \\
  \hline\hline
  Markov constants & $2\sqrt{10}$ & $2\sqrt{5}$ & $4(13+8\sqrt{2})$ & $\sqrt{13}$ \\
  \hline
   Polarization constants & 3 & 2 & $2+\frac{\sqrt{2}}{2}$ & $\frac{3}{2}$ \\
   \hline
   Unconditional Constants & 2 & 3 & $5+4\sqrt{2}$ & 5 \\
  \hline
\end{tabular}
\bigskip
\end{center}
Furthermore, all the constants appearing in the previous table are sharp. Actually, the extreme polynomials where the constants are attained are the following:
\begin{enumerate}
\item $\pm(x^2+y^2-6xy)$ for the simplex.

\item $\pm(x^2+y^2-4xy)$ for the sector $D\left(\frac{\pi}{2}\right)$.

\item $\pm\left(x^2+(5+4\sqrt{2})y^2-(4+4\sqrt{2})xy\right)$ for the sector $D\left(\frac{\pi}{4}\right)$.

\item $\pm(x^2+y^2-3xy)$ for the unit square.
\end{enumerate}

Compare the previous table with similar results that hold for 2-homogeneous polynomials on the Banach spaces $\ell_1^2$, $\ell_2^2$ and $\ell_\infty^2$:

\begin{center}
\bigskip
\renewcommand{\arraystretch}{1.5}
\begin{tabular}{|c||c|c|c|}
  \hline
   & ${\mathcal P}(^2\ell_1^2)$ & ${\mathcal P}\left(^2\ell_2^2\right)$ & ${\mathcal P}(^2\ell_\infty^2)$  \\
  \hline\hline
  Markov constants & $4$ & $2$ & $2\sqrt{2}$ \\
  \hline
   Polarization constants & 2 & 1 & $2$ \\
   \hline
   Unconditional Constants & $\frac{1+\sqrt{2}}{2}$ & $\sqrt{2}$ & $1+\sqrt{2}$ \\
  \hline
\end{tabular}
\bigskip
\end{center}
Observe that the Markov constants of the spaces ${\mathcal P}(^2\ell_1^2)$ and ${\mathcal P}(^2\ell_\infty^2)$ can be calculated taking into consideration the description of the geometry of those spaces given in \cite{CK3}. Also, the Markov constant of ${\mathcal P}(^2\ell_2^2)$ is twice its polarization constant, or in other words, 2.

\noindent On the other hand, the constants appearing in the second line of the previous table are well-known results (see for instance \cite{IS}).

\noindent Finally, the unconditional constants corresponding to the third line of the previous table were calculated in Theorem 3.5, Theorem 3.19 and Theorem 3.6 of \cite{GMS2}.

\bibliographystyle{amsplain}

\end{document}